\documentclass[]{amsart}

\usepackage{amsmath}
\usepackage{amsthm}
\usepackage{amssymb}
\usepackage{amsfonts}
\usepackage{bm}
\usepackage[bookmarks=true,hyperindex,pdftex,colorlinks,citecolor=red, linkcolor=cyan]{hyperref}
\usepackage{bbm}
\usepackage{enumitem}
\usepackage{mathtools}
\usepackage{marginnote}
\usepackage{xcolor}
\usepackage[normalem]{ulem}
\usepackage{tikz-cd}

\DeclareMathOperator{\diam}{diam}

\DeclareMathOperator{\dom}{dom}

\newcommand{\N}{\mathbb{N}}                                 
\newcommand{\R}{\mathbb{R}}                                 
\newcommand{\BST}{\mathcal{ST}}                             

\newcommand{\set}[1]{\left\{{#1}\right\}}                   
\newcommand{\norm}[1]{\left\|{#1}\right\|}                  

\newcommand{\closedball}[1]{B_{#1}}                         
\newcommand{\fmol}[1]{f_{#1}}                               

\definecolor{egraf}{rgb}{0.2,0.4,0}

\newcommand{\restricted}{\mathord{\upharpoonright}}

\newcommand{\eps}{\varepsilon}

\newcommand{\abs}[1]{\left\vert#1\right\vert}
\newcommand{\cardinality}[1]{\abs{#1}}

\renewcommand{\d}[3]{d_{{#2}}^{#3}({#1})}

\newcommand{\Dinf}{D_\infty^\omega}
\newcommand{\dinf}{d_\infty}

\newcommand{\Dinfip}{D^{(i,+)}_{\infty}}
\newcommand{\Dinfim}{D^{(i,-)}_{\infty}}

\newcommand{\ta}{t_\alpha}
\newcommand{\ba}{b_\alpha}
\newcommand{\Da}{D_\alpha}
\newcommand{\da}{d_\alpha}
\newcommand{\ftba}{\fmol{\ta,\ba}}

\newcommand{\xa}[1]{x_\alpha^{#1}}                      
\newcommand{\DapUn}{D^{(1,+)}_{\alpha}}
\newcommand{\DamUn}{D^{(1,-)}_{\alpha}}

\newcommand{\tbn}{t_{\beta_n}}
\newcommand{\bbn}{b_{\beta_n}}

\newcommand{\Db}{D_\beta}
\newcommand{\Dbn}{D_{\beta_n}}
\newcommand{\db}{d_\beta}

\newcommand{\Ta}{T_\alpha}
\newcommand{\Tb}{T_\beta}
\newcommand{\Tbn}{T_{\beta_n}}
\newcommand{\dat}{$\delta$-$\alpha$-tree}

\newcommand{\dast}{$\delta$-$\alpha$-sprawling tree}

\newcommand{\RT}{\mathcal{T}}                           
\newcommand{\RTa}{\mathcal{T}_\alpha}
\newcommand{\RTb}{\mathcal{T}_\beta}
\newcommand{\dab}{$\delta$-$\alpha$-bush}
\newcommand{\dbb}{$\delta$-$\beta$-bush}

\theoremstyle{plain}
\newtheorem{theorem}{Theorem}[section]
\newtheorem*{theorem*}{Theorem}
\newtheorem{lemma}[theorem]{Lemma}
\newtheorem{corollary}[theorem]{Corollary}
\newtheorem*{corollary*}{Corollary}
\newtheorem{proposition}[theorem]{Proposition}

\theoremstyle{definition}
\newtheorem*{definition*}{Definition}
\newtheorem{definition}[theorem]{Definition}
\newtheorem{example}[theorem]{Example}
\newtheorem{question}[theorem]{Question}
\newtheorem*{question*}{Question}

\theoremstyle{remark}
\newtheorem{remark}[theorem]{Remark}

\begin{document}
	
	\title[On metric characterizations of tree indices of Banach spaces]{On metric characterizations of tree and fragmentability indices of Banach spaces}
	
	\author[E. Basset]{Estelle Basset}
	\address[E. Basset]{Universit\'e Marie et Louis Pasteur, CNRS, LmB (UMR 6623), F-25000 Besan\c con, France.}
	\email{estelle.basset@univ-fcomte.fr}

	\author[G. Lancien]{Gilles Lancien}
	\address[G. Lancien]{Universit\'e Marie et Louis Pasteur, CNRS, LmB (UMR 6623), F-25000 Besan\c con, France.}
	\email{gilles.lancien@univ-fcomte.fr}

	\author[A. Proch\'azka]{Anton\'in Proch\'azka}
	\address[A. Proch\'azka]{Universit\'e Marie et Louis Pasteur, CNRS, LmB (UMR 6623), F-25000 Besan\c con, France.}
	\email{antonin.prochazka@umlp.fr}

	\date{}
	
	
	\begin{abstract} We introduce two ordinal indices that are linear invariants for Banach spaces: the dyadic tree index and the sprawling tree index. We show that they are also bi-Lipschitz invariants. In fact, we characterize their values in terms of sub-Lipschitz embeddability of dyadic or countably branching diamond graphs of ordinal height. We derive applications for separable Banach spaces that are universal for complete countable metric spaces and bi-Lipschitz embeddings. We also discuss the links of these tree indices with classical fragmentability indices of Banach spaces such as the dentabilty, weak fragmentability and Szlenk indices. 
		
	\end{abstract}
	
	\maketitle
	
	\section{Introduction}
	
	One of the goals of the non-linear geometry of Banach spaces is to show that some linear properties of Banach spaces are stable under certain non-linear embeddings. An even more ambitious aim is to characterize these linear properties in purely metric terms, that is by properties involving only the distance and forgetting about the vector space operations. 
	In this paper, the non-linear maps that we shall consider will essentially be the bi-Lipschitz embeddings. In this spirit, our work is an attempt to metrically characterize some quantifications of properties like the Radon-Nikod\'{y}m Property (RNP for short) or the Point of Continuity Property (PCP for short). The dentability index $D(X)$ of a Banach space $X$ is an ordinal, that is non decreasing under linear embeddings and  measures how strong is the RNP of this Banach space, or how far it is from failing it. For the PCP, the relevant index is the weak fragmentability index of $X$, denoted $\Phi(X)$. We refer to Section \ref{s:FragmentabilityIndices} for the definitions. 
	
	In \cite{Basset2025}, the first named author constructed diamond metric graphs of countable ordinal height and studied their Lipschitz-free space. In particular she showed that they could have arbitrarily high dentability or weak fragmentability indices. 
	The initial aim of this work was to study the bi-Lipschitz embeddability of these diamond graphs into general Banach spaces, 
	as opposed to
	their canonical embedding into their Lipschitz-free space, 
	seeking 
	metric characterizations of the values of these indices. 
	It turns out that the relevant  properties to consider are closely related but different from 
	the RNP and 
	the PCP.
	Namely, the Infinite Tree Property and what we call the Infinite Sprawling Tree Property. We introduce in Section \ref{s:TreeIndices} the associated indices $DT(X)$ and $ST(X)$. Using tools from 
	descriptive set theory 
	and adapting a fundamental argument of J. Bourgain \cite{Bourgain}, we show that  a separable Banach space $X$ has the Infinite Tree Property if and only if $DT(X)<\omega_1$, and that it has the Infinite Sprawling Tree Property if and only if $ST(X)<\omega_1$, where $\omega_1$ is the first uncountable ordinal. In Section \ref{s:diamonds} we recall the constructions of the dyadic diamond graphs $D_\alpha^2$ and the countably branching diamond graphs $D_\alpha^\omega$, for $\alpha$ countable ordinal. We also extend to these graphs the notions of active pairs and sub-Lipschitz embeddings due to M. I. Ostrovskii \cite{Ostrovskii}. 
	
	After introducing these objects, the main findings of this paper are proved in Sections \ref{s:dyadictrees} and \ref{s:embeddingCountableDiamond}. They are the following.
	
	\begin{theorem}\label{t:A} Let $X$ be a separable Banach space and $\alpha \in (0,\omega_1)$. Assume that $Y$ is a Banach space and that $Y$ bi-Lipschitz embeds into $X$. Then 
		\begin{enumerate}
			\item $DT(X)>\alpha$ if and only if $D_\alpha^2$ sub-Lipschitz embeds into $X$.
			\item $DT(Y)\le DT(X)$
			\item If $X$ fails the Infinite Tree property, so does $Y$. 
		\end{enumerate}
	\end{theorem}
	
	\begin{theorem}\label{t:B} Let $X$ be a separable Banach space and $\alpha \in (0,\omega_1)$. Assume that $Y$ is a Banach space and that $Y$ bi-Lipschitz embeds into $X$. Then 
		\begin{enumerate}
			\item $ST(X)>\alpha$ if and only if $D_\alpha^\omega$ sub-Lipschitz embeds into $X$.
			\item $ST(Y)\le ST(X)$
			\item If $X$ fails the Infinite Sprawling Tree property, so does $Y$. 
		\end{enumerate}
	\end{theorem}
	
	In the course of the paper, we introduce yet another index -- the convex weak fragmentability index $C(X)$ of a Banach space $X$ -- and the related Banach space property that we call Slice PCP. This property is characterized, for separable Banach spaces, by the condition $C(X)<\omega_1$. The Slice PCP is strictly weaker than RNP and formally stronger than PCP, although we don't know whether 
	they are different. 
	We show that $ST(X)\le C(X)$ and deduce from Theorem \ref{t:B} that a separable Banach space $X$ such that for all $\alpha<\omega_1$, $D_\alpha^\omega$ sub-Lipschitz embeds into $X$ must fail Slice PCP. In particular, a separable Banach space, which is universal for countable complete metric spaces and bi-Lipschitz embeddings must fail Slice PCP. 
	
	In the final Section \ref{Diamonds_L1}, we show that for all $\alpha \in (0,\omega_1)$, $D_\alpha^\omega$ bi-Lipschitz embeds into $L_1$ with distortion 2. As a consequence, a separable  Banach space which contains a bi-Lipschitz copy of $L_1$ must fail Slice PCP. This result is new and cannot be obtained by differentiability arguments.

	\section{Various fragmentability indices and their relations to some Banach spaces properties}\label{s:FragmentabilityIndices}
	
	Let us first introduce some basic notation. 
	All Banach spaces will be over the field $\R$. 
	Given two Banach spaces $X$ and $Y$, we denote $B(X,Y)$ the space of all bounded linear  operators from $X$ to $Y$, and we write $X^* = B(X, \R)$ the topological dual of $X$. We denote the closed unit ball of $X$ by $B_X$, and its unit sphere by $S_X$. 
	
	We denote $\omega$ the first infinite ordinal, which, as it is usual, will also be identified with $[0,\omega)=\N \cup \{0\}$, the set of finite ordinals. Similarly, $\omega_1$ is the first uncountable ordinal, that can be identified with $[0,\omega_1)$, the set of countable ordinals. 
	
	We now describe a general peeling scheme which is frequently used to assign some isomorphically invariant ordinal index to a given Banach space $X$. Assume that $\mathcal A$ is a collection of open subsets of $X$ that is stable under dilations and translations and let $C$ be a subset of $X$. 
	For $\varepsilon>0$ we define the set derivation
	\[
	[C]_\varepsilon'=C \setminus \bigcup \set{A \in {\mathcal A}: \diam(A \cap C)<\varepsilon}
	\]
	and we put
	\[
	[C]_\varepsilon^0:=C, \quad [C]_\varepsilon^{\alpha+1}:=[[C]_\varepsilon^\alpha]_\varepsilon' \quad\mbox{ and }\quad [C]_\varepsilon^{\beta}:=\bigcap_{\alpha<\beta} [C]_\varepsilon^\alpha
	\]
	for every ordinal $\alpha$ and every limit ordinal $\beta$.
	Further we define 
	\[
	\iota_{\mathcal A}(X,\varepsilon):=\inf \set{\alpha: [\closedball{X}]_\varepsilon^\alpha=\emptyset} \mbox{ and } \iota_{\mathcal A}(X):=\sup_{\varepsilon>0} \iota_{\mathcal A}(X,\varepsilon),
	\]
	adopting the convention that $\inf \emptyset = \infty$ and $\alpha < \infty$ for every ordinal $\alpha$. It is clear that $\iota_{\mathcal A}(X)<\infty$ if and only if, for every bounded subset $C$ of $X$ and every $\varepsilon>0$, there exists $A \in \mathcal A$ such that $A \cap C \neq \emptyset$ and $\diam (A \cap C)<\varepsilon$. Then, we say that $X$ is \emph{fragmentable by $\mathcal A$}. Then the value of $\iota_{\mathcal A}(X)$  measures how fast or efficiently this procedure operates and we call it the \emph{$\mathcal A$-fragmentability index of $X$}.
	Naturally, if $\mathcal A$ is a class stable under linear isomorphisms, then $\iota_{\mathcal A}(X)$ is an isomorphic invariant. 
	
	\medskip
	We now detail the three main examples that will be considered in this paper. 
	\begin{enumerate}
		\item If we consider $\mathcal A$ to be the collection of weak open subsets of a Banach space $X$, the index $\iota_{\mathcal A}(X)$ is called the \emph{weak fragmentability index of $X$} and for $C$ a subset of $X$, we denote $\sigma_\varepsilon^\alpha(C):=[C]_\varepsilon^\alpha$. Then $\Phi(X,\varepsilon):=\iota_{\mathcal A}(X,\varepsilon)$ and $\Phi(X):=\iota_{\mathcal A}(X)$. The index $\Phi(X)$ is clearly an isomorphic invariant. 
		\item We call an \emph{open half-space} of a Banach space $X$ a set of the form $S=\{x\in X,\ x^*(x)>a\}$, with $x^*\in X^*$ and $a\in \R$. 
		If $\mathcal A$ is the collection of open half-spaces of $X$, then $\iota_{\mathcal A}(X)$ is called the \emph{dentability index} of $X$ and for $C$ a subset of $X$, we denote $d_\varepsilon^\alpha(C):=[C]_\varepsilon^\alpha$. 
		Then $D(X,\varepsilon):=\iota_{\mathcal A}(X,\varepsilon)$ and $D(X):=\iota_{\mathcal A}(X)$. 
		Again, the index $D(X)$ is  an isomorphic invariant and, since open half-spaces are weakly open, $\Phi(X)\le D(X)$. 
		\item If $X=Y^*$ is a dual Banach space and $\mathcal A$ is the collection of weak$^*$ open sets of $Y^*$, the index $\iota_{\mathcal A}(X)$ is called the \emph{Szlenk index of $Y$} and for $C$ subset of $Y^*$, we denote $s_\varepsilon^\alpha(C):=[C]_\varepsilon^\alpha$. Then $Sz(Y,\varepsilon):=\iota_{\mathcal A}(X,\varepsilon)$ and $Sz(Y):=\iota_{\mathcal A}(X)$. It is important to note that $Sz(Y)$ is an isomorphic invariant for $Y$, but not for $X$, which explains the notation. Since weak$^*$ open sets in $Y^*$ are weakly open, we have that $\Phi(X)\le Sz(Y)$. 
	\end{enumerate}
	
	If $C$ is a convex subset of $X$, then $d_\eps^1(C)$ is still convex. This is not the case for the derivations defining the weak fragmentability index and the Szlenk index. We shall need convexifications of these indices. For that purpose, we introduce two new derivations and associated indices.
	\begin{enumerate}
		\item Let $C$  be a closed bounded convex subset of $X$ and $\eps>0$. We define $\gamma_\eps^1(C)$ to be the closed convex hull of $\sigma_\eps^1(C)$ and inductively $\gamma_\eps^{\alpha+1}(C)=\gamma_\eps^1(\gamma_\eps^\alpha(C))$ and $\gamma_\eps^\alpha(C)=\cap_{\beta<\alpha}\gamma_\eps^\beta(C)$, if $\alpha$ is a limit ordinal. It is clear that the sets $\gamma_\eps^\alpha(C)$ are all closed and convex.  
		Then $C(X,\eps)=\inf\{\alpha,\ \gamma_\eps^\alpha(B_X)=\emptyset\}$ and $C(X)=\sup_{\eps>0}C(X,\eps)$. We call $C(X)$ the \emph{convex weak-fragmentability index} of $X$. This is again an isomorphic invariant. It is clear that 
		$$\Phi(X)\le C(X)\le D(X).$$
		
		\item The dual version of this procedure is more classical and was introduced in \cite{GKL2001}. We recall its definition. Let $K$ be a convex weak$^*$-compact subset of a dual Banach space $X^*$ and $\eps>0$. Then $c_\eps^1(K)$ is the weak$^*$-closed convex hull of $s_\eps^1(K)$. We define analogously $c_\eps^\alpha(K)$ for $\alpha$ ordinal. Then $Cz(X,\eps)=\inf\{\alpha,\ c_\eps^\alpha(B_{X^*})=\emptyset\}$ and $Cz(X)=\sup_{\eps>0}Cz(X,\eps)$. The index $Cz(X)$ is called the \emph{convex Szlenk index} of $X$. 
	\end{enumerate}
	
	For any of the indices that we have introduced, the fact that its value is countable characterizes an important isomorphic property for separable Banach spaces. We refer to Section 2.4 in \cite{BLP1} for a more detailed presentation. Here, we just give the necessary classical definitions and the corresponding characterizations. Let us first recall that a Banach space $X$ has the \emph{Radon-Nikod\'{y}m Property} (RNP in short) if and only if every non empty bounded subset $C$ of $X$ is dentable, i.e. for every $\eps>0$, there exists an open half space $S$ in $X$ such that $S\cap C \neq \emptyset$ and $\diam (S \cap C)<\eps$. Since a separable Banach space has a countable basis of open sets, the following is then easy. 
	
	\begin{proposition}\label{prop:RNP_D(X)} Let $X$ be a separable Banach space. Then $X$ has the RNP if and only if $D(X)<\omega_1$.
	\end{proposition}
	
	We now recall that a Banach space $X$ has the \emph{Point of Continuity Property} (PCP in short) if and only if every non empty bounded closed subset $C$ of $X$ is weakly fragmentable, meaning that for every $\eps>0$, there exists a weakly open subset $V$ of $X$ such that $V\cap C \neq \emptyset$ and $\diam (V \cap C)<\eps$. Similarly, we have: 
	
	\begin{proposition}
		Let $X$ be a separable Banach space. Then $X$ has the PCP if and only if $\Phi(X)<\omega_1$.
	\end{proposition}
	
	We now recall a classical property of the Szlenk index (see Theorem 1 in \cite{Lancien2006}).
	\begin{proposition} Let $X$ be a separable Banach space. Then the following are equivalent
		\begin{enumerate}
			\item $X^*$ has the RNP,
			\item $X^*$ is separable,
			\item $Sz(X)<\omega_1$. 
		\end{enumerate}
	\end{proposition}
	The following is proved in \cite{LPR} and will be relevant for this work. 
	
	\begin{theorem}\label{t:LPR} Let $X$ be a separable Banach space. Then $Sz(X)=Cz(X)$. In particular $X^*$ is separable if and only if $Cz(X)<\omega_1$. 
	\end{theorem}
	
	In order to characterize the condition $C(X)<\omega_1$, we need to introduce a new property. 
	\begin{definition} Let $X$ be a Banach space. We say that $X$ has the \emph{Slice Point of Continuity Property} (Slice PCP) if for any bounded subset $A$ of $X$ and any $\eps >0$, there exists an open half-space $S$  such that $S\cap A \neq \emptyset$ and any $x\in S \cap A$ admits a weak neighborhood $V$ satisfying $\diam(V\cap A)<\varepsilon$. 
	\end{definition}
	
	The Slice PCP and the condition $C(X)<\omega_1$ are related as follows.
	\begin{proposition}\label{p:SlicePCP}
		Let $X$ be a separable Banach space. Then $C(X)<\omega_1$ if and only if $X$ has the Slice PCP. 
	\end{proposition}
	
	\begin{proof}
		Assume first that $X$ has the Slice PCP and let $C$ be a closed bounded convex non empty subset of $X$ and $\eps>0$. Then there exists an open half-space $S$  such that $S \cap C \neq \emptyset$ and any $x\in S\cap C$ admits a weak neighborhood $V$ satisfying $\diam(V\cap C)<\varepsilon$. Thus $\sigma_\eps^1(C)\subset C\setminus S$. Since $C \setminus S$ is closed and convex, we deduce that $\gamma_\eps^1(C)\subset C\setminus S \neq C$. Since $X$ has a countable basis of norm open sets, it readily follows that for all $\eps>0$, $C(X,\eps)<\omega_1$ and thus that $C(X)<\omega_1$. 
		
		Assume now that $X$ fails the Slice PCP. So, there exists a bounded subset $A$ of $X$ and $\eps >0$ such that for any  open half-space  $S$ with $S\cap A\neq \emptyset$ we have $\sigma_\eps^1(A)\cap S \neq \emptyset$. 
		We may assume that $A \subset B_X$ and we let $C$ be the closed convex hull of $A$.
		We claim that $\gamma_\eps^1(C)=C$; this implies that for all ordinals $\alpha$, $C\subset \gamma_\eps^\alpha(B_X)$ and $C(X)=\infty$.
		So, let us assume that there is $x \in C \setminus \gamma_\eps^1(C)$.
		Let $T$ be an open half-space such that $x \in T\cap C$ and $T \cap \gamma_\eps^1(C)=\emptyset$. 
		It is then easy to see that $T\cap A$ is non empty. 
		Thus $\emptyset \neq \sigma_\eps^1(A)\cap T \subset \sigma_\eps^1(C)\cap T \subset \gamma_\eps^1(C)\cap T$ which is the desired contradiction which proves our claim. 
	\end{proof}
	
	In order to give interesting examples of spaces with the Slice PCP, we need to recall the following definition.
	
	\begin{definition} Let $(X,\|\ \|)$ be a Banach space. We say that $\|\ \|$ is asymptotically uniformly convex (AUC) if for any $\eps>0$, there exists $\delta>0$, such that for all $x\in X$ such that $\|x\|> 1-\delta$, we have that $x$ admits a weak neighborhood $V$ so that $\diam(V\cap B_X)<\eps$.
	\end{definition}
	
	\begin{proposition}\label{p:AUC implies SlicePCP} Let $X$ be a Banach space with an equivalent AUC norm. Then $C(X)\le \omega$ and $X$ has the Slice PCP. 
	\end{proposition}
	
	\begin{proof} Assume, as we may, that the norm of $X$ is AUC. Let $\eps>0$. There exists $\delta>0$, such that for all $x\in X$ such that $\|x\|> 1-\delta$, we have that $x$ admits a weak neighborhood $V$ so that $\diam(V\cap B_X)<\eps$. It follows that $\sigma_\eps^1(B_X)\subset (1-\delta)B_X$ and thus $\gamma_\eps^1(B_X)\subset (1-\delta)B_X$. Then, by a simple rescaling argument $\gamma_\eps^n(B_X)\subset (1-\delta)^nB_X$, for all $n\in \N$. Thus, for $n$ large enough, $\diam (\gamma_\eps^n(B_X))<\eps$ and $\gamma_\eps^{n+1}(B_X)=\emptyset$. Therefore, for all $\eps>0$, $C(X,\eps)<\omega$, where $\omega$ is the first infinite ordinal, and $C(X)\le \omega$. It follows from the previous proposition, and more precisely from the implication that does not use the separability assumption, that $X$ has the Slice PCP. 
	\end{proof}
	
	\begin{remark} It is obvious that RNP implies Slice PCP and that Slice PCP implies PCP. We do not know whether PCP implies Slice PCP. However there exists a Banach space with Slice PCP and failing RNP. Indeed, the predual of the James Tree space is known to fail RNP \cite{LindenstraussStegall}. On the other hand, it has been shown by M. Girardi  \cite{Girardi} that it has an equivalent AUC norm. So, by Proposition \ref{p:AUC implies SlicePCP}, it has Slice PCP. 
	\end{remark}
	
	
	\section{Tree indices and related Banach space properties}\label{s:TreeIndices}
	
	In this section, we introduce two new ordinal indices, related to the existence of generalized dyadic trees and what we will call sprawling trees (hoping this image will help the intuition). We start with a construction ``by hands'' of these trees, as they will be instrumental for proving our results on the existence of new metric invariants. We conclude the section with a more global approach, that will allow us to use tools from descriptive set theory.

	Before to start, we need to introduce more notation. We let $\omega^{<\omega}$ be the set of finite sequences in $\omega=\N \cup \{0\}$, including the empty sequence $\varnothing$. Similarly $2^{<\omega}$ is the set of finite sequences in $\{0,1\}$. For $s=(s_1,\ldots,s_n) \in \omega^{<\omega}$, $|s|=n$ is the \emph{length} of $s$, and
	$|\varnothing|=0$. The natural order on $\omega^{<\omega}$ is denoted $\preceq$ and for $s,t \in \omega^{<\omega}$, $s \preceq t$, if $t$ extends $s$. 
	
	
	\subsection{Generalized dyadic trees}
	\label{s: delta-alpha-trees}
	We recall that a subset of $T$ of $\omega^{<\omega}$ is a \emph{tree} on $\omega$ if it satisfies the following property: for any $t\in T$ and any $s\preceq t$, we have $s\in T$. We now define a special family $(T_\alpha)_{\alpha<\omega_1}$ of trees on $\omega$ as follows.
	
	\begin{itemize}
		\item $T_0 =\set{\varnothing}$.
		\item $T_{\alpha+1}=\set{\varnothing} \cup 0^\smallfrown T_\alpha \cup 1^\smallfrown T_\alpha$ (notice that $s^\smallfrown \varnothing=s$) for every $\alpha$ such that $T_\alpha$ has been constructed. 
		\item Let $\alpha<\omega_1$ be a limit ordinal and fix $(\beta_n)_{n=0}^\infty$  an enumeration of all ordinals less than $\alpha$. Then $T_\alpha=\set{\varnothing}\cup \bigcup\limits_{n=0}^\infty n^\smallfrown T_{\beta_n}$. 
	\end{itemize}
	Our arguments below are independent of the particular enumeration fixed at the third item above. We abusively call these trees \emph{dyadic trees}, although they may be countably branching at some vertices. However, most vertices have two successors. The following obvious lemma describes this. 
	
	\begin{lemma}\label{l:dyadictrees}
		For every $\alpha<\omega_1$ and every $s \in T_\alpha$, there are three possibilities:
		\begin{itemize}
			\item $s$ has 0 successors (in other words $s$ is maximal in $T_\alpha$ for $\preceq$). \item $s$ has exactly 2 successors (this is the case in particular when $s=\varnothing$ and $\alpha$ is a successor ordinal).
			\item The set of successors of $s$ is $\{s^\smallfrown n,\ n \in \omega\}$ (this is the case in particular when $s=\varnothing$ and $\alpha$ is a limit ordinal).
		\end{itemize}
	\end{lemma}
	
	We now turn to the definition of a {\dat} in a Banach space. 
	
	\begin{definition}
		\label{def: dat}
		Let $X$ be a Banach space. For $\delta > 0$ and $\alpha< \omega_1$, we call \textit{\dat}\ of $X$ a family $(x_s)_{s \in \Ta} \subset X$ such that, for every $s \in \Ta$, the following conditions are satisfied.
		\begin{enumerate}
			\item[$\bullet$] if $s$ has two successors $s^\smallfrown 0$ and $s^\smallfrown 1$ in $\Ta$, then $x_s = \frac12(x_{s^\smallfrown 0} + x_{s ^\smallfrown 1})$ with $\norm{x_{s^\smallfrown 0} - x_{s}}= \norm{x_{s^\smallfrown 1} - x_{s}}=\frac12\norm{x_{s^\smallfrown 0} - x_{s^\smallfrown 1}} \geq \delta$.
			
			\item[$\bullet$] if $s$ has infinitely countably many successors $s^\smallfrown n$ in $\Ta$, then $x_s = x_{s^\smallfrown n}$ for every $n \in \omega$.
		\end{enumerate}
	\end{definition}
	
	The above gives naturally rise to the following family of isomorphic properties.
	
	\begin{definition}\label{d:dyadic-tree-index} 
		Let $X$ be a separable Banach space and $\alpha \in [0,\omega_1)$. We say that $X$ has the \emph{$\alpha$-tree property} if there exists $\delta>0$ such that $B_X$ contains a {\dat}. Then we define the \emph{dyadic-tree index} of $X$ as follows. For $\delta>0$, $DT(X,\delta)=\inf\{\alpha<\omega_1,\ B_X\ \text{does not contain a {\dat}}\}$ if it exists and $DT(X,\delta)=\infty$ otherwise, and $DT(X)=\sup\{DT(X,\delta),\ \delta>0\}$. 
	\end{definition}
	
	We conclude with a simple link between the dentability index and the dyadic-tree index. 
	
	\begin{proposition}
		\label{prop: link between dentability index and d-a-t}
		Let $X$ be a Banach space. Then $DT(X)\le D(X)$. More precisely, if $DT(X,\delta)>\alpha$, for $\delta>0$ and $\alpha<\omega_1$, then $D(X,\delta)>\alpha$. 
	\end{proposition}
	
	\begin{proof} So assume that $B_X$ contains a \dat. We will show that its  root $x_\varnothing$ belongs to $\d{B_X}{\delta}{\alpha}$. We proceed by induction on $\alpha$. The statement is clear if $\alpha = 0$, so now assume it is true for all $\gamma <\alpha$.
		
		If $\alpha = \beta+1$, let $(x_s)_{s \in \Ta} \subset B_X$ be a \dat. 
		Let $\Tb^0$ (resp. $\Tb^1$) be the set of all 
		$0^\smallfrown t$ 
		(resp. $1^\smallfrown t$) 
		with $t \in \Tb$. 
		It is easy to check that $(x_s)_{s \in \Tb^0}$ and $(x_s)_{s \in \Tb^1}$ are  $\delta$-$\beta$-trees, with respective roots $x_0$ and $x_1$. So, by induction hypothesis, $x_0,x_1 \in \d{B_X}{\delta}{\beta}$. Since $\d{B_X}{\delta}{\beta}$ is convex, $x_\varnothing = \frac12(x_0+x_1) \in \d{B_X}{\delta}{\beta}$. 
		Now, if $S$ is an open slice containing $x_\varnothing$, $S$ contains at least $x_0$ or $x_1$. 
		But $\norm{x_\varnothing - x_0}=\norm{x_\varnothing - x_1} = \frac{1}{2}\norm{x_0-x_1} \geq \delta$, so $\diam(S \cap \d{B_X}{\delta}{\beta}) \geq \delta$ and $x_\varnothing \in \d{B_X}{\delta}{\alpha}$.
		
		If $\alpha$ is a limit ordinal, let $(x_s)_{s \in \Ta} \subset B_X$ be a \dat. 
		For every $n \in \omega$, $x_\varnothing = x_n$, but $x_n$ is the root of the $\delta$-$\beta_n$-tree $(x_{n \smallfrown t})_{t \in \Tbn}$, where $(\beta_n)_n$ is an enumeration of all ordinals less than $\alpha$. So by induction hypothesis, $x_\varnothing \in \d{B_X}{\delta}{\beta_n}$. 
	\end{proof}
	
	\begin{remark}\label{r:FTP1} Note that a Banach space $X$ has the classical Finite Tree Property (FTP) if and only if, there exists $\delta>0$ such that $B_X$ contains a $\delta$-$n$-tree for all $n\in \omega$. In that case, using a translation, we see that $2B_X$ contains a $\delta$-$n$-tree with root $0$ and, after dividing it by 2, $B_X$ contains a $\frac{\delta}{2}$-$n$-tree with root $0$, for all $n\in \omega$. Thus $B_X$ contains a $\frac{\delta}{2}$-$\omega$-tree with root $0$ and $DT(X)>\omega$. On the other hand, the fact that $DT(X)>\omega$ implies FTP is clear. We recall that a Banach space is super-reflexive if and only if it fails the FTP, a result due to R.C. James \cite{James}. So it is equivalent to $DT(X)\le \omega$. By Enflo's theorem \cite{Enflo}, this is also equivalent to the fact that $X$ admits an equivalent uniformly convex norm and  equivalent to $D(X)\le \omega$ (see \cite{Lancien1995}). As we will see in Section \ref{sss:delta-alpha-trees}, $D(X)$ and $DT(X)$ are no longer equal for higher values. In fact the conditions $D(X)<\omega_1$ and $DT(X)<\omega_1$ are not equivalent for separable Banach spaces.
	\end{remark}
	

	\subsection{Sprawling trees}\label{s:sprawlingtrees}
	
	We introduce now the $\alpha$-sprawling tree-property. First we define the notion of $\alpha$-sprawling trees. 
	
	The set $(\set{0,1}\times\omega)^{<\omega}$, of all finite sequences in $\set{0,1}\times\omega$ is equipped with natural extension order, still denoted $\preceq$. For every $\alpha \in [0,\omega_1)$, we build inductively a tree $S_\alpha \subset (\set{0,1}\times\omega)^{<\omega}$ as follows.
	\begin{itemize}
		\item $S_0 =\set{\varnothing}$.
		\item $S_{\alpha+1}=\set{\varnothing} \cup \bigcup_{n\in \omega} \left((0,n)^\smallfrown S_\alpha \cup (1,n)^\smallfrown S_\alpha\right)$ (notice that $s^\smallfrown \varnothing=s$) for every $\alpha$ such that $S_\alpha$ has been constructed.
		\item If $\alpha$ is a limit ordinal, we set $S_\alpha=\set{\varnothing}\cup \bigcup_{n \in \omega} (0,n)^\smallfrown S_{\beta_n}$, where $(\beta_n)_{n=0}^\infty$ is an enumeration of all the ordinals les than $\alpha$. 
	\end{itemize}
	
	The first lemma follows from a straightforward transfinite induction.
	\begin{lemma}
		\label{l: nb of successors in a sprawling tree}
		For every $\alpha \in [0,\omega_1)$ and every $s \in S_\alpha$, there are three possibilities:
		\begin{itemize}
			\item[(A)] $s$ has 0 successors (or $s$ is maximal in $(S_\alpha,\preceq)$). 
			\item[(B)] $s$ has infinitely many successors of type $s^\smallfrown(0,n)$ and infinitely many successors of type $s^\smallfrown(1,n)$ (this is the case in particular when $s=\varnothing$ and $\alpha$ is a successor ordinal)
			\item[(C)] $s$ has infinitely many successors of type of type $s^\smallfrown(0,n)$ and no successors of type $s^\smallfrown(1,n)$ (this is the case in particular when $s=\varnothing$ and $\alpha$ is a limit ordinal). 
		\end{itemize}
	\end{lemma}
	
	We now turn to the definition of a \dast\ in a Banach space. 
	\begin{definition}
		For $\delta > 0$ and $\alpha \in [0, \omega_1)$, we call \textit{\dast}\ of $X$ a family $(x_s)_{s \in S_\alpha} \subset X$ such that, for every $s \in S_\alpha$, the following conditions are satisfied:
		\begin{enumerate}
			\item[$\bullet$] if $s$ satisfies (B), then
			\[x_s = \frac12(x_{s^\smallfrown (0,n)} + x_{s^\smallfrown (1,n)}),\  n\in \omega\ \text{and}\ \|x_{s^\smallfrown (0,n)} - x_{s^\smallfrown (0,m)}\| \geq \delta,\ n\neq m;
			\]
			\item[$\bullet$] if $s$ satisfies (C), then $x_s = x_{s^\smallfrown (0,n)}$ for all  $n \in \omega$. 
		\end{enumerate}
		A $\delta$-1-sprawling tree will be called a \emph{$\delta$-spider}.
	\end{definition}
	
	Notice that if $s$ satisfies (B),  then 
	$x_{s^\smallfrown (0,n)} - x_{s^\smallfrown (0,m)}
	=x_{s^\smallfrown (1,
		m)} - x_{s^\smallfrown (1,
		n)}$ 
	and  \[\Big\|x_s-\frac12(x_{s^\smallfrown(0,n)}+x_{s^\smallfrown(1,m)})\Big\|=\frac12\|x_{s^\smallfrown(1,n)}-x_{s^\smallfrown(1,m)}\|\geq \frac\delta2,
	\]
	for all $n\neq m$.
	
	\medskip Similarly as for the dyadic trees, we may now define a related family of isomorphic properties.
	
	\begin{definition}\label{d:ST-index}
		Let $X$ be a separable Banach space and $\alpha \in [0,\omega_1)$. We say that $X$ has the \emph{$\alpha$-sprawling tree property} if there exists $\delta>0$ such that $B_X$ contains a {\dast}. Then we define the \emph{sprawling-tree index} of $X$ as follows. For $\delta>0$, $ST(X,\delta)=\inf\{\alpha<\omega_1,\ B_X\ \text{does not contain a {\dast}}\}$ if it exists and $ST(X,\delta)=\infty$ otherwise, and $ST(X)=\sup\{ST(X,\delta),\ \delta>0\}$. 
	\end{definition}

	The following lemma is the essence of the arguments in \cite{Basset2025}. The idea can be traced back at least to the paper~\cite{LLT} by Lin, Lin and Troyanski. 
	\begin{lemma}\label{l:LLT}
		Let $X$ be a Banach space, $x \in X$ and $(x_i)_{i=0}^\infty$, $(y_i)_{i=0}^\infty$ be  bounded sequences in $X$ such that $x=\frac12(x_i+y_i)$ for all $i\in \omega$. Then, for every weak neighborhood $V$ of $x$ there exists an infinite subset $I$ of $\omega$ such that $\frac12(x_i+y_j) \in V$ for all $i,j\in I$. In particular $x$ belongs to the weak closure of  $\{\frac12(x_i+y_j):\ i\neq j \in \omega\}$.
	\end{lemma}
	
	\begin{proof}
		Let $V$ be a weak neighborhood of $x$. We may assume that $V$ is given by conditions $x_1^*<\alpha_1,\ldots, x_n^*<\alpha_n$, with $x_1^*,\ldots,x_n^* \in X^*$ and $\alpha_1,\ldots,\alpha_n \in \R$.  Since $(x_i)_i$ and $(y_i)_i$ are bounded, by passing to a subsequence, we may assume that
		$\lim_{i \to \infty}\langle x^*_k, x_i\rangle = c_k$ and $\lim_{i \to \infty}\langle x^*_k, y_i\rangle = d_k$. We have that $\frac12(c_k+d_k)=\langle x^*_k,x\rangle$. It follows that for $i, j$ sufficiently large $\langle x^*_k,\frac12(x_i+y_j) \rangle < \alpha_k$.
	\end{proof}
	
	With this lemma in hands, we can relate the $\alpha$-sprawling tree property to the index $C(X)$. 
	
	\begin{proposition}\label{p:C(X)_ST(X)}
		Let $X$ be a Banach space. Then $ST(X)\le C(X)$. More precisely, if $ST(X,\delta)>\alpha$, for $\delta>0$ and $\alpha<\omega_1$, then $C(X,\delta)>\alpha$.
	\end{proposition}
	
	\begin{proof} So assume that $B_X$ contains a $\delta$-$\alpha$-sprawling tree $(x_s)_{s\in S_\alpha}$. We will show by induction that $x_\varnothing \in \gamma_{\delta}^\alpha(B_X)$. 
		If $\alpha=0$, there's nothing to prove. Assume now that we have proved the claim for every $\gamma<\alpha$. We recall that the labels (A), (B), (C) in Lemma \ref{l: nb of successors in a sprawling tree} refer to the three possible structures of the set of successors of an element of the tree $S_\alpha$. 
		
		If $\alpha=\beta+1$ then $\varnothing$ satisfies $(B)$. For all $i\in \omega$, $x_{(0,i)}$ and $x_{(1,i)}$ are roots of a $\delta$-$\beta$-sprawling tree in $B_X$. So, by induction hypothesis, they belong to $\gamma_{\delta}^\beta(B_X)$. Let now $V$ be a weak neighborhood of $x_\varnothing$. By Lemma~\ref{l:LLT} there are $i\neq j$ such that $$\frac{x_{(0,i)}+x_{(1,j)}}2\in V \mbox{ and } \frac{x_{(0,j)}+x_{(1,i)}}2\in V.$$ 
		We have 
		\[
		\Big\|\frac{x_{(0,i)}+x_{(1,j)}}2-\frac{x_{(0,j)}+x_{(1,i)}}2\Big\|=\|x_{(0,i)}-x_{(0,j)}\|\geq \delta.
		\]
		It follows from the convexity of $\gamma_\delta^\beta(B_X)$ that $\diam(V \cap \gamma_{\delta}^\beta(B_X))\geq \delta$.
		
		Finally, when $\alpha$ is a limit ordinal, the claim follows by induction hypothesis as, for every $n\in \omega$, we have $x_\varnothing=x_{(0,n)}$ and $(x_{(0,n)^\smallfrown s})_{s\in T_{\beta_n}} \subset B_X$ is a $\delta$-$\beta_n$-sprawling tree in $B_X$. So $x\in \bigcap_{n\in \omega}\gamma_{\delta}^{\beta_n}(B_X)=\gamma_{\delta}^\alpha(B_X)$. 
	\end{proof}
	
	As we will see in the next section the reverse inequality does not hold and the conditions $ST(X)<\omega_1$ and $C(X)<\omega_1$ are not equivalent for separable Banach spaces. 
	
	\subsection{A more global approach and applications} In this section, we revisit the indices $DT(X)$ and $ST(X)$ with the tools from descriptive set theory. For that purpose, we take inspiration from the fundamental paper by J. Bourgain \cite{Bourgain}, where ordinal indices  measuring the linear embeddability of a given Banach space with a basis were introduced. 
	
	We start with a very general setting, that we will soon specialize to our two indices. So, let us fix $A=(A_n)_{n=0}^\infty$  a strictly increasing sequence of countable sets, with $A_0 \neq \emptyset$. We denote $A_\infty:=\bigcup_{n=0}^\infty A_n$. Let $X$ be a separable Banach space.
	For any map $f:B\to C$, we will use the notation $\dom(f):=B$. In what follows we will always consider the topology of pointwise convergence on sets of functions $X^B$.
	We say that $T$ is an \emph{$A$-tree on $X$} if
	\begin{itemize}
		\item $T\subset \bigcup_{n=0}^\infty X^{A_n}$
		\item for every $f\in T$, $\dom(f)=A_n \implies f\restricted_{A_{n-1}}\in T$.
	\end{itemize}
	
	
	We say that an $A$-tree $T$ on $X$ is \emph{closed} if, for every $n\in \omega$, the set $T\cap X^{A_n}$ is closed in $X^{A_n}$ equipped with the topology of pointwise convergence.
	
	An $A$-tree $T$ on $X$ is \emph{ill founded} if there exists $f\in X^{A_\infty}$ such that $f\restricted_{A_n} \in T$ for every $n \in \omega$.
	Otherwise, we say that $T$ is \emph{well founded}. 
	
	Let $T$ be an $A$-tree on $X$. An element $f\in T$ is \emph{maximal in $T$} if $T$ contains no proper extension of $f$.
	For $\alpha$ ordinal, we define inductively the $A$-tree $T^0=T$, 
	\begin{align*}
		T^{\alpha+1}&=\set{\mbox{non-maximal elements of }T^\alpha}\\
		&=\bigcup_{n=0}^\infty \set{f\restricted_{A_n}:f\in T^{\alpha} \mbox{ and } \dom(f)=A_{n+1}},
	\end{align*}
	and $T^\alpha=\bigcap_{\beta<\alpha}T^\beta$, if $\alpha$ is a limit ordinal. 
	Then the \emph{height} of $T$ is $h(T)=\inf\{\alpha,\ T^\alpha=\emptyset\}$ if it exists and $h(T)=\infty$ otherwise.
	
	For every $n \in \omega$ we denote $\pi_n:X^{A_{n+1}}\to X^{A_n}$ the map defined by $\pi_n(f)=f\restricted_{A_n}$, for $f\in X^{A_{n+1}}$. Our first statement is an adaptation of Lemma 2 in \cite{Bourgain}.
	
	\begin{lemma}\label{l:illfounded} 
		Let $X$ be a Banach space and $T$ be a closed non empty $A$-tree on $X$ such that $T\cap X^{A_n}=\overline{\pi_n(T\cap X^{A_{n+1}})}$ for all $n\in \omega$. Then $T$ is ill founded.
	\end{lemma}
	
	\begin{proof} 
		Since $T$ is a non empty $A$-tree, there exists $f_0\in T\cap X^{A_0}$. 
		Let us denote $A_0=\set{a_0^0,a_0^1,\ldots}$.
		Using that $T\cap X^{A_0}=\overline{\pi_0(T\cap X^{A_{1}})}$ 
		we can find $f_1 \in T\cap X^{A_1}$ such that $\norm{f_0(a_0^0)-f_1(a_0^0)}\leq 2^{-0}$.
		Next $A_1=A_0\cup \set{a_1^0,a_1^1,\ldots}$, where this union is disjoint.
		Using that $T\cap X^{A_1}=\overline{\pi_0(T\cap X^{A_{2}})}$ 
		we can find $f_2 \in T\cap X^{A_2}$ such that $\norm{f_1(a_i^j)-f_2(a_i^j)}\leq 2^{-1}$
		for every $j\leq i\leq 1$.
		Following in the obvious fashion, we obtain a sequence $(f_n)$ such that $f_n\in T\cap X^{A_n}$ and such that for every $a \in A_\infty$, the sequence $(f_n(a))_n$ is eventually defined and Cauchy in $X$.
		Let us define $f(a)$ as the limit of this sequence yielding $f\in X^{A_\infty}$.
		Now, for every $n\in \omega$, $f\restricted_{A_n}\in T$ as $T\cap X^{A_n}$ is closed.
		This proves that $T$ is ill founded.   
	\end{proof}
	
	Next we adapt Proposition 3 from \cite{Bourgain}.
	
	\begin{lemma}\label{l:WFomega_1} 
		Assume that $X$ is separable and that $T$ is a well founded closed $A$-tree on $X$. 
		Then $h(T)<\omega_1$. 
	\end{lemma}
	
	\begin{proof} 
		Note that for any ordinal $\alpha$, $T^\alpha$ is an $A$-tree on $X$ and, by definition, $T^{\alpha+1}\cap X^{A_n}=\pi_n(T^\alpha\cap X^{A_{n+1}})$. 
		Since the family $(T^\alpha)_\alpha$ is decreasing and, for all $n\in \omega$, $X^{A_n}$ has a countable basis of open sets, there exists $\alpha_0<\omega_1$ such that $\overline{T^{\alpha}\cap X^{A_n}}=\overline{T^{\alpha+1}\cap X^{A_n}}$, for all $n\in \omega$ and all $\alpha\ge \alpha_0$. Consider now  $S=\bigcup_{n=0}^\infty \overline{T^{\alpha_0}\cap X^{A_n}}$. 
		Clearly $S$ is a closed $A$-tree on $X$. 
		Since $T$ is closed, we have that $S \subset T$. 
		Therefore $S$ is well founded. 
		Finally,  for all $n\in \omega$, we have
		\[
		\begin{aligned}
			S\cap X^{A_n}
			&=\overline{T^{\alpha_0}\cap X^{A_n}}
			=\overline{T^{\alpha_0+1}\cap X^{A_n}}
			=\overline{\pi_n(T^{\alpha_0}\cap X^{A_{n+1}})}\\
			&=\overline{\pi_n(\overline{T^{\alpha_0}\cap X^{A_{n+1}}})}
			=\overline{\pi_n(S\cap X^{A_{n+1}})}.
		\end{aligned}
		\]
		So, by Lemma~\ref{l:illfounded}, $S=\emptyset$. 
		It follows that $T^{\alpha_0}=\emptyset$. This finishes the proof. 
	\end{proof}
	
	In what follows we are going to consider two concrete realizations of the above abstract scheme. One linked to $\delta$-$\alpha$-trees the other to $\delta$-$\alpha$-sprawling trees.
	
	\subsubsection{$\delta$-$\alpha$-trees.}\label{sss:delta-alpha-trees} 
	We let  $A=\bigcup_{n=0}^\infty A_n$, with  $A_n=\set{0,1}^{\leq n}$. Notice that for $n\in \omega$, $A_n=T_n$, where the trees $(T_\alpha)_{\alpha <\omega_1}$ have been defined in Section \ref{s: delta-alpha-trees}. Let $X$ be a separable Banach space and $\delta>0$. We denote $\mathcal{DT}(X,\delta)$ the set of all $f \in \bigcup_{n=0}^\infty (B_X)^{A_n}$ satisfying: for all $n\in \N$, all $f\in \mathcal{DT}(X,\delta)\cap (B_X)^{A_n}$ and all $s\in A_{n-1}$, we have 
	\[f(s)=\frac12(f(s^\smallfrown 0)+f(s^\smallfrown 1))\ \text{and}\ \frac12\|f(s^\smallfrown 0)-f(s^\smallfrown 1)\|\ge \delta.
	\]
	
	The following lemma is obvious.
	
	\begin{lemma}\label{l:DT closed}
		$\mathcal{DT}(X,\delta)$ is a closed $A$-tree on $X$.
	\end{lemma}
	
	We now recall the following classical definition.
	\begin{definition} We say that a Banach space $X$ has the \emph{Infinite Tree Property} if  there exists $\delta>0$ and a map $f:A_\infty \to B_X$ such that for every $n\in \omega$, $f\restricted_{A_n} \in \mathcal{DT}(X,\delta)$.
	\end{definition}
	So, the following is a tautology. 
	
	\begin{lemma}\label{l:ITP_illfounded} 
		Let $X$ be a Banach space. 
		Then $X$ has the Infinite Tree Property if and only if there exists $\delta >0$ such that $\mathcal{DT}(X,\delta)$ is ill founded. 
	\end{lemma}
	
	Recall that $DT(X,\delta)$ stands for the dyadic-tree index introduced in Section \ref{s: delta-alpha-trees}. 
	
	\begin{proposition}\label{p:DT=DT}
		Let $X$ be a separable Banach space and $\delta>0$. Then $$DT(X,\delta)=h(\mathcal{DT}(X,\delta)).$$ 
	\end{proposition}
	
	\begin{proof} First, 
		we will show by induction on $\alpha<\omega_1$ that, for all $k\in \omega$, if $(x_s)_{s\in T_{\alpha+k}}$ is a $\delta$-$(\alpha+k)$-tree in $B_X$, then $f:A_k \to B_X$ defined by $f(s):=x_s$ for $s\in A_k$, belongs to $\mathcal{DT}(X,\delta)^\alpha$.
		In particular, if $(x_s)_{s\in T_{\alpha}}$ is a $\delta$-$\alpha$-tree in $B_X$, then $(\varnothing \mapsto x_\varnothing) \in \mathcal{DT}(X,\delta)^\alpha$. This will prove that $h(\mathcal{DT}(X,\delta))\geq DT(X,\delta)$.
		
		For $\alpha=0$, 
		it is obvious
		that whenever $(x_s)_{s\in T_{k}}$ is a $\delta$-$k$-tree in $B_X$, 
		then $f=(s\mapsto x_s)_{s\in A_k}$ belongs to $\mathcal{DT}(X,\delta)$ (remember that $A_k=T_k$).
		
		Next, let $\alpha \in (0,\omega_1)$ and assume the statement true for all $\mu<\alpha$. 
		Assume first that $\alpha=\beta+1$ and let $(x_s)_{s\in T_{\alpha+k}}$ be a $\delta$-$(\alpha+k)$-tree in $B_X$. 
		This is also a $\delta$-$(\beta+k+1)$-tree in $B_X$. 
		So by our induction hypothesis, $f=(s\mapsto x_s)_{s\in A_{k+1}} \in \mathcal{DT}(X,\delta)^\beta$. Therefore $f\restricted_{A_k} \in \mathcal{DT}(X,\delta)^{\beta+1}=\mathcal{DT}(X,\delta)^{\alpha}$.
		
		Assume now that  $\alpha$ is a limit ordinal, and let $(\beta_n)_{n\in \omega}$ be an enumeration of all ordinals less than $\alpha$. Let $k\in \omega$ and $(x_s)_{s\in T_{\alpha+k}}$ be a $\delta$-$(\alpha+k)$-tree in $B_X$. 
		Then $(y_s)_{s\in T_{\beta_n+k}}$ defined as
		\[
		y_s=\begin{cases}x_s& \mbox{if } \abs{s}\leq k,\\
			x_{s\restricted_k^\smallfrown (n)^\smallfrown t}& \mbox{if } s=s\restricted_k^\smallfrown t \mbox{ with }\abs{t}\geq 1,
		\end{cases}
		\]
		is a $\delta$-$(\beta_n+k)$-tree in $B_X$.
		Thus the induction hypothesis implies that
		$f=(s\mapsto x_s)_{s\in A_k}\in \mathcal{DT}(X,\delta)^{\beta_n}$, for all $n\in \omega$. So $f\in \mathcal{DT}(X,\delta)^{\alpha}$.

		For the other inequality, we show by induction on $\alpha<\omega_1$ that for all $k\in \omega$,
		if $f \in \mathcal{DT}(X,\delta)^\alpha$ with $\dom(f)=A_k$,
		then there exists a $\delta$-$(\alpha+k)$-tree $(y_s)_{s\in T_{\alpha+k}}$ in $B_X$ such that for all $s \in A_k$, $y_s=f(s)$. In particular, if $(\varnothing \mapsto x)\in \mathcal{DT}(X,\delta)^\alpha$, then $x$ is the root of a $\delta$-$\alpha$-tree in $B_X$. 
		This will show that $h(\mathcal{DT}(X,\delta))\leq DT(X,\delta)$.
		
		For $\alpha=0$, it is clear that if $f\in \mathcal{DT}(X,\delta)$ with $\dom(f)=A_k$, then setting $y_s:=f(s)$ for $s\in T_k=A_k$, we have that $(y_s)_{s\in T_{k}}$ is a $\delta$-$k$-tree in $B_X$. 
		
		Next, let $\alpha \in (0,\omega_1)$ and assume that our statement is true for all $\mu<\alpha$. 
		Assume first that $\alpha=\beta+1$ and let $f \in \mathcal{DT}(X,\delta)^\alpha=
		\mathcal{DT}(X,\delta)^{\beta+1}$ with $\dom(f)=A_k$. 
		There exists $g \in \mathcal{DT}(X,\delta)^\beta$ with $\dom(g)=A_{k+1}$ and $g\restricted_{A_k}=f$. 
		So, by our induction hypothesis, there exists a $\delta$-$(\beta+k+1)$-tree 
		$(y_s)_{s\in T_{\beta+k+1}}$ in $B_X$ such that for all $s\in A_{k+1}$, $y_s=g(s)$. Then $(y_s)_{s\in T_{\alpha+k}}$ is a $\delta$-$(\alpha+k)$-tree in $B_X$, and $y_s=f(s)$, for all $s \in A_k$.
		
		Assume now that  $\alpha$ is a limit ordinal, and let $(\beta_n)_{n\in \omega}$ be an enumeration of all ordinals less than $\alpha$. 
		Let $f \in \mathcal{DT}(X,\delta)^\alpha=\bigcap_{n\in \omega}\mathcal{DT}(X,\delta)^{\beta_n}$ with $\dom(f)=A_k$. 
		By our induction hypothesis, for all $n\in \omega$, there exists a $\delta$-$(\beta_n+k)$-tree $(y_s^n)_{s\in T_{\beta_n+k}}$ in $B_X$ such that $y_s^n=f(s)$, for all $s\in A_k$. 
		We now define 
		\[
		y_s:=
		\begin{cases}
			f(s), &\mbox{for }  s\in A_k,\\
			f(t), &\mbox{for } s=t^\smallfrown (n) \mbox{ with } t\in A_k \mbox{ and } \abs{t}=k,\\
			y_{t^\smallfrown u}^n, &\mbox{for } s=t^\smallfrown (n)^\smallfrown u \mbox{ with } t\in A_k, \abs{t}=k \mbox{  and } u\in T_{\beta_n}.
		\end{cases}
		\]
		It is easy to check that $(y_s)_{s\in T_{\alpha+k}}$ is a $\delta$-$(\alpha+k)$-tree in $B_X$. This finishes our inductive proof.
	\end{proof}
	
	We can now state and prove the desired characterization of spaces failing the infinite tree property.
	
	\begin{theorem}\label{t:ITP-DT(X)} Let $X$ be a separable Banach space. Then $X$ fails the Infinite Tree Property if and only if $DT(X)<\omega_1$. 
	\end{theorem}
	
	\begin{proof}
		Assume first that $X$ has the infinite tree property. Then there exists $\delta>0$ such that $h(\mathcal{DT}(X,\delta))=DT(X,\delta)=\infty$. 
		
		Assume now that $X$ fails the Infinite Tree Property. Then, by Lemma \ref{l:ITP_illfounded}, $\mathcal{DT}(X,\delta)$ is well founded for all $\delta>0$. Since $X$ is separable, by Lemmas \ref{l:WFomega_1} and \ref{l:DT closed}, $DT(X,\delta)=h(\mathcal{DT}(X,\delta))<\omega_1$, for all $\delta>0$. It follows that $DT(X)=\sup_{\delta >0}DT(X,\delta)<\omega_1$. 
	\end{proof}
	
	\begin{example} As a first consequence, we obtain that the dentability index and the dyadic-tree index do not coincide in general. Indeed, the Bourgain-Rosenthal space $X_{BR}$, constructed in \cite{BourgainRosenthal}, is a separable Banach space failing both the RNP and the Infinite Tree Property. Thus $D(X_{BR})=\infty$, while $DT(X_{BR})<\omega_1$. 
	\end{example}
	
	\subsubsection{$\delta$-$\alpha$-sprawling trees.} A similar procedure can be applied for sprawling trees. We shall only indicate the main steps and the small differences. This time, we let  $A_k:=(\set{0,1}\times\omega)^{\leq k}$. Notice that $A_k=S_k$ for every $k<\omega$ where the trees $S_\alpha$ have been defined in Section \ref{s:sprawlingtrees}. 
	Let $X$ be a separable Banach space and $\delta>0$. 
	We denote $\BST(X,\delta)$ the set of all $f \in \bigcup_{k=0}^\infty (B_X)^{A_k}$ satisfying: for all $k\in \N$, all $f\in \mathcal{ST}(X,\delta)\cap (B_X)^{A_k}$, and all $s\in A_{k-1}$, we have 
	\begin{gather*}
		f(s) = \frac12(f(s^\smallfrown (0,n)) + f(s^\smallfrown (1,n))),\  n\in \omega\ \\
		\text{and}\\ 
		\|f(s^\smallfrown (0,n)) - f(s^\smallfrown (0,m))\| \geq \delta,\ n\neq m \in \omega.
	\end{gather*}
	
	\begin{lemma} \label{l:ST-closed}
		$\BST(X,\delta)$ is a closed $A$-tree on $X$.
	\end{lemma}
	
	\begin{definition} We say that a Banach space $X$ has the \emph{Infinite Sprawling Tree property} (ISTP) if there exists $\delta>0$ and a map $f:A_\infty \to B_X$ such that for every $n\in \omega$, $f\restricted_{A_n} \in \BST(X,\delta)$.
	\end{definition}
	
	\begin{lemma}\label{l:ISFTP_illfounded} 
		Let $X$ be a Banach space. 
		Then $X$ has the Infinite Sprawling Tree Property if and only if there exists $\delta >0$ such that $\BST(X,\delta)$ is ill founded. 
	\end{lemma}

	\begin{proposition}\label{ST=ST}
		Let $X$ be a 
		separable Banach space and $\delta>0$. Then $$ST(X,\delta))=h(\BST(X,\delta)).$$ 
	\end{proposition}

	\begin{proof} We follow the scheme of the proof of Proposition \ref{p:DT=DT}. 
		
		To prove that $h(\BST(X,\delta))\geq ST(X,\delta)$, we show by induction on $\alpha<\omega_1$ that, for all $k\in \omega$, $f:A_k \to B_X$ defined by $f(s):=x_s$ for $s\in A_k$, belongs to $\mathcal{ST}(X,\delta)^\alpha$. The arguments for the case $\alpha=0$ and the induction step for $\alpha$ successor ordinal are identical. So assume $\alpha$ is a limit ordinal and that the statement is true for all $\mu<\alpha$. Let $(\beta_n)_{n\in \omega}$ be an enumeration of all ordinals less than $\alpha$. Let $k\in \omega$ and $(x_s)_{s\in S_{\alpha+k}}$ be a $\delta$-$(\alpha+k)$ sprawling tree in $B_X$. 
		Then $(y_s)_{s\in S_{\beta_n+k}}$ defined as
		\[
		y_s=\begin{cases}x_s& \mbox{if } \abs{s}\leq k,\\
			x_{s\restricted_k^\smallfrown (0,n)^\smallfrown t}& \mbox{if } s=s\restricted_k^\smallfrown t \mbox{ with }\abs{t}\geq 1,
		\end{cases}
		\]
		is a $\delta$-$(\beta_n+k)$ sprawling tree in $B_X$.
		Thus the induction hypothesis implies that
		$f=(s\mapsto x_s)_{s\in S_k}=(s\mapsto y_s)_{s \in S_{k}}\in \mathcal{ST}(X,\delta)^{\beta_n}$, for all $n\in \omega$. So $f\in \BST(X,\delta)^{\alpha}$.
		
		Next, we show by induction on $\alpha<\omega_1$ that for all $k\in \omega$,
		if $f \in \BST(X,\delta)^\alpha$ with $\dom(f)=S_k$,
		then there exists a $\delta$-$(\alpha+k)$ sprawling tree $(y_s)_{s\in S_{\alpha+k}}$ in $B_X$ such that for all $s \in A_k=S_k$, $y_s=f(s)$. This will show that $h(\BST(X,\delta))\leq ST(X,\delta)$. Again, the arguments for the case $\alpha=0$ and the induction step for $\alpha$ successor ordinal are the same as for Proposition \ref{p:DT=DT}. So assume $\alpha$ is a limit ordinal and that the statement is true for all $\mu<\alpha$. Let $(\beta_n)_{n\in \omega}$ be an enumeration of all ordinals less than $\alpha$. Let $f \in \BST(X,\delta)^\alpha=\bigcap_{n\in \omega}\BST(X,\delta)^{\beta_n}$ with $\dom(f)=S_k$. By our induction hypothesis, for all $n\in \omega$, there exists a $\delta$-$(\beta_n+k)$ sprawling tree $(y_s^n)_{s\in S_{\beta_n+k}}$ in $B_X$ such that $y_s^n=f(s)$, for all $s\in S_k$. 
		We now define 
		\[
		y_s:=
		\begin{cases}
			f(s), &\mbox{for }  s\in S_k,\\
			f(t), &\mbox{for } s=t^\smallfrown (0,n) \mbox{ with } t\in S_k \mbox{ and } \abs{t}=k,\\
			y_{t^\smallfrown u}^n, &\mbox{for } s=t^\smallfrown (0,n)^\smallfrown u \mbox{ with } t\in S_k, \abs{t}=k \mbox{  and } u\in S_{\beta_n}.
		\end{cases}
		\]
		It is easy to check that $(y_s)_{s\in S_{\alpha+k}}$ is a $\delta$-$(\alpha+k)$ sprawling tree in $B_X$. This finishes our inductive proof. 
	\end{proof}
	
	Similarly, we deduce:

	\begin{theorem}\label{t:ISTP_ST(X)}
		Let $X$ be a separable Banach space. Then $X$ fails the Infinite Sprawling Tree Property if and only if $ST(X)<\omega_1$.
	\end{theorem}
	
	\begin{example} V. Kadets and D. Werner proved in \cite{KadetsWerner} that a minor refinement of the Bourgain-Rosenthal space $X_{BR}$, that we shall denote $X_{KW}$, fails the Infinite Tree Property but has the Daugavet Property. On the other hand, a Banach space with the Daugavet property fails PCP (see Theorem~3.2.1 in~\cite{KMRW25}). Therefore $\Phi(X_{KW})=C(X_{KW})=\infty$, while $ST(X_{KW})\le DT(X_{KW})<\omega_1$. 
	\end{example}

	
	\subsection{Dentability index and generalized bushes}\label{s:bushes}
	
	It is well known (see \cite{DiestelUhl}) that the Radon Nikod\'{y}m property and the dentability can be restated in terms of bushes or martingales but, as we have just seen, not always in terms of dyadic trees. In this section we give a quantitative version of this by relating the dentability index with the existence of what we will call the $\alpha$-bush property. The content of this section is essentially known by specialists, but cannot be found in the literature. We take this opportunity to state and prove it. This will also allow us to later formulate a related nonlinear question: see Question \ref{q:embeddingbushes}. 
	
	For that purpose, we define inductively for every $\alpha \in[0,\omega_1)$ a family $\RTa$ of trees in $\N^{<\omega}$ of height $\alpha$, in which each node has either finitely many or infinitely countably many successors.
	
	\begin{itemize}
		\item $\RT_0 = \set{\{\varnothing\}}$.
		
		\item If $T$ is a tree in $\omega^{<\omega}$, then $T \in \RT_{\alpha+1}$ if and only if there exist $n \in \N$ and $T_0, \ldots, T_n \in \RTa$ such that $T = \set{\varnothing}\cup \bigcup_{i =0}^n i^\smallfrown T_i$. 
		
		\item If $\alpha$ is a limit ordinal and $T$ is a tree in $\omega^{<\omega}$, then $T \in \RT_{\alpha}$ if and only if there exist an enumeration $(\beta_n)_{n\in \omega}$ of all ordinals less than $\alpha$ and, for every $n \in \omega$, a tree $T_{\beta_n} \in \RT_{\beta_n}$ such that $T = \set{\varnothing}\cup \bigcup_{n \in \omega} n^\smallfrown T_{\beta_n}$.
	\end{itemize}

	Next, we give the definition of a $\delta$-$\alpha$-bush. 
	\begin{definition}
		\label{def: dab}
		For $\delta > 0$ and $\alpha \in [0,\omega_1)$, we call \textit{\dab}\ in $X$ a family $(x_s)_{s \in T} \subset X$ where $T \in \RTa$ and such that, for every $s \in T$, the following conditions are satisfied:
		\begin{itemize}
			\item if the set $s^+$ of successors of $s$ is finite and nonempty, then there exists $(\lambda_t)_{t\in s+}\subset (0,1]$ such that $x_s = \sum_{t \in s^+} \lambda_t x_t$,  $\sum_{t \in s^+} \lambda_t = 1$, and $\norm{x_s - x_t} \ge \delta$, for all $t \in s^+$.
			
			\item if $s$ has infinitely countably many successors $s^\smallfrown n$, $n \in \omega$, then $x_s = x_{s^\smallfrown n}$ for every $n \in \omega$.
		\end{itemize}
		
	\end{definition}
	Note that $\delta$-$\alpha$-trees are particular examples of \dab es. 
	
	\begin{definition}
		We say that a Banach space $X$ has the \textit{$\alpha$-bush-property} if there exists some $\delta > 0$ such that $B_X$ contains a \dab.
	\end{definition}
	
	We can now state the following characterization.  
	
	\begin{theorem} Let $X$ be a Banach space and $\alpha \in [0,\omega_1)$. Then $X$ has the $\alpha$-bush-property if and only if $D(X) > \alpha$.
	\end{theorem}
	
	We state and prove each implication in two separate propositions. 
	
	\begin{proposition} Let $X$ be a Banach space such that $B_X$ contains a \dab\ $(x_s)_{s \in T}$ for some $\alpha \in [0, \omega_1)$, $T \in \RTa$ and $\delta > 0$. Then,  $x_\varnothing \in \d{B_X}{\delta}{\alpha}$, and thus $D(X) > \alpha$.
	\end{proposition}
	
	\begin{proof} We proceed by induction. The statement is obviously true for $\alpha=0$. Assume it is true for all $\gamma<\alpha$ and that $(x_s)_{s \in T}$ is a \dab \ in $B_X$. 
		
		If $\alpha$ is a limit ordinal, and $(\beta_n)_{n\in \omega}$ is an enumeration of all ordinals less than $\alpha$  such that  $T = \set{\varnothing}\cup \bigcup_{n \in \omega} n^\smallfrown T_{\beta_n}$. Then, for every $n \in \omega$, $x_\varnothing = x_n$ is the root of a $\delta$-$\beta_n$-bush. Applying the induction hypothesis, we obtain that 
		$x_\varnothing \in \bigcap_{n \in \omega} \d{B_X}{\delta}{\beta_n} = \d{B_X}{\delta}{\alpha}.$
		
		Assume now that $\alpha = \beta +1$ is a successor ordinal. We can write $x_\varnothing= \sum_{i=0}^r \lambda_i x_i$ where $r\in \N$, $\{(0),\ldots,(r)\}$ is the set of immediate successors of $\varnothing$ in $T$, $\lambda_0, \ldots, \lambda_r \in (0,1]$, $\sum_{i=0}^r \lambda_i = 1$, and $\|x_\varnothing - x_i\| \geq \delta$ for all $i \in \set{0, \ldots, r}$. For all $i\le r$, $x_i$ is the root of a \dbb\ in $B_X$ and, by induction hypothesis, $x_i \in \d{B_X}{\delta}{\beta}$. Since $\d{B_X}{\delta}{\beta}$ is convex it also contains $x_\varnothing$. Let now $S$ be an open slice containing $x_\varnothing$. The complementary of $S$ being convex, there exists $i \in \set{0, \cdots, r}$ such that $x_i \in S$, and then $\norm{x_\varnothing - x_i} \ge \delta$, so $\diam(S \cap \d{B_X}{\delta}{\beta}) \ge  \delta$. This shows that $x_\varnothing \in \d{B_X}{\delta}{\alpha}$. 
		
	\end{proof}
	
	To prove the converse, we shall need the following classical consequence of the Hahn-Banach separation theorem.
	
	\begin{lemma}
		\label{lemma: HB separation thm}
		Let $C$ be a bounded closed convex set. Then for every $\varepsilon > 0$, for every $x \in \d{C}{\varepsilon}{}$, we have that $x$ belongs to the closed convex hull of $C \setminus B(x, \frac{\varepsilon}{2})$. 
	\end{lemma}
	
	\begin{proposition}
		Let $X$ be a Banach space such that $D(X) > \alpha$ for some $\alpha \in [0, \omega_1)$. Then, $X$ has the $\alpha$-bush property.  
	\end{proposition}
	
	\begin{proof} It will clearly be enough to prove by induction on $\alpha \in [0,\omega_1)$ the following statement  that we denote $(H_\alpha)$: 
		for all $\eps \in (0,1)$, $\eta\in (0,1)$, $M>1$, and $x\in \d{B_X}{\varepsilon}{\alpha}$, there exists an $\frac{\varepsilon}{3}$-$\alpha$-bush in $MB_X$, whose root $x_\varnothing$ satisfies $\norm{x_\varnothing - x} < \eta$. 
		
		The statement $(H_0)$ is obvious: just set $x_\varnothing=x$. So let $\alpha \in (0,\omega_1)$ and assume that $(H_\beta)$ is true for all $\beta<\alpha$. 
		
		Suppose first that $\alpha=\beta+1$ and fix $\eps \in (0,1)$, $\eta\in (0,1)$, $M>1$, and $x\in \d{B_X}{\varepsilon}{\beta+1}$. Assume, as we may that $\eta<\frac\eps6$. By Lemma \ref{lemma: HB separation thm}, there exist $n\in \N$, $x_0, \ldots, x_n \in \d{B_X}{\varepsilon}{\beta}$ and $\lambda_0, \ldots, \lambda_n \in (0,1]$ such that 
		\[\sum_{i=0}^n \lambda_i = 1,\ \Big\|x - \sum_{i=0}^n \lambda_i x_i\Big\|<\frac\eta3,\ \text{and}\  \|x-x_i\|>\frac\eps2,\ \text{for all}\ i\in \{0,\ldots,n\}.\]
		The induction hypothesis implies that, for every $i \in \set{0, \ldots, n}$, there exist $T_i \in \RTb$ along with $(x_t^i)_{t \in T_i}$ an $\frac{\varepsilon}{3}$-$\beta$-bush in $MB_X$ such that $\|x_\varnothing^i - x_i\|<\frac\eta3$. Let us set $T \coloneqq \set{\varnothing}\cup \bigcup_{i = 0}^n (i)^\smallfrown T_i$, $z_\varnothing \coloneqq \sum_{i=0}^n \lambda_i x_\varnothing^i$ and, for $s \in T \setminus \set{\varnothing}$, $z_s \coloneqq x_t^i$ for $i \in \set{0, \ldots, n}$ and $t \in T_i$ such that $s = i^\smallfrown t$. It is clear that $T \in \RTa$ and $(z_s)_{s \in T} \subset MB_X$. It follows from the convexity of the norm that $\|x-z_\varnothing\|<\frac{2\eta}{3}$. From this, we deduce that 
		\[\norm{z_\varnothing - x_\varnothing^i} \geq \norm{x - x_i} - \norm{x - z_\varnothing} - \norm{x_\varnothing^i - x_i}> \frac{\varepsilon}{2} - \eta>\frac\eps3.
		\]
		Therefore $(z_s)_{s \in T}$ is an $\frac{\varepsilon}{3}$-$\alpha$-bush. This finishes the successor case. 
		
		Assume now that  $\alpha$ is a limit ordinal and let $(\beta_n)_{n\in \omega}$ be an enumeration of all ordinals less than  $\alpha$. Fix $\eps \in (0,1)$, $\eta\in (0,1)$, $M>1$, and $x\in \d{B_X}{\varepsilon}{\alpha}=\bigcap_{n=0}^\infty \d{B_X}{\varepsilon}{\beta_n}$. Assume, as we may, that $\eta<\frac12(M-1)$. By induction hypothesis, for every $n \in \omega$, there exists an $\frac{\varepsilon}{3}$-$\beta_n$-bush $(x_t^n)_{t \in T_n}$ in $\frac12(M+1)B_X$ such that $\|x-x_\varnothing^n\|<\eta$. We set $T \coloneqq \set{\varnothing}\cup \bigcup_{n \in \omega} (n)^\smallfrown T_n$, $z_\varnothing=x$, and $z_s \coloneqq x_t^n + x - x_\varnothing^n$ for $n \in \omega$, $t \in T_n$ such that $s = n^\smallfrown t$. Note that $z_n = x_\varnothing^n + x - x_\varnothing^n = x = z_\varnothing$, for all $n\in \omega$. It is clear that $T \in \RTa$ and that for all $s\in T$, $\|z_s\|\le \frac12(M+1)+\eta \le M$. Thus $(z_s)_{s\in T}$ is an $\frac{\varepsilon}{3}$-$\alpha$-bush rooted at $x$ and included in $MB_X$. This finishes the limit case. 
	\end{proof}
	
	

	
	\section{Diamond graphs and sub-Lipschitz embeddings.}\label{s:diamonds}
	We now introduce the metric spaces and related notions of embeddings that will be studied throughout the remainder of this paper. 
	First, we recall the construction of diamond graphs of ordinal height introduced in~\cite{Basset2025}.
	\begin{definition}\label{d:Diamonds}
		Let $1\leq \kappa \leq \omega$. We are going to define inductively the \emph{$\kappa$-branching diamond graphs} $D_\alpha^\kappa$, for $\alpha \in [0,\omega_1)$ as follows. We set  $D_0^\kappa=\{t_0,b_0\}$, with $d_0^\kappa(t_0,b_0)=1$. Then, we define  $D_1^\kappa=\set{t_1,b_1} \cup \set{x_i:0\leq i< \kappa}$. We consider the complete bipartite graph on these vertices (edges connect every vertex in the first set to every vertex in the second set). The edge $\set{t_1,x_i}$ will be denoted $(i,+)$ and the edge $\set{b_1,x_i}$ will be denoted $(i,-)$. 
		We define $d=d_1=d_1^\kappa$ on $D_1^\kappa$ as $\frac12$ times the shortest path metric corresponding to this graph structure, i.e. $d(t_1,b_1)=1$.
		The points $t_1, b_1$ are called \emph{top} and \emph{bottom}, respectively. 
		Similarly for the points $t_\alpha, b_\alpha$ defined below.
		
		When $D_\alpha^\kappa$ has been defined we define $D_{\alpha+1}^\kappa$ as the weighted graph, together with its shortest path metric, where each edge in $D_1^\kappa$ (that we call original $D_1^\kappa$) has been replaced by $D_\alpha^\kappa$. This copy is called $D_\alpha^{\kappa(j,\pm)}$ where $(j,\pm)$ corresponds to the position of the edge being replaced. When there is no ambiguity about the number of branchings, we will drop $\kappa$ and simply write $D_\alpha^{(j,\pm)}$. The metric $d_{\alpha+1}=d_{\alpha+1}^\kappa$ restricted to any of these copies of $D_\alpha^\kappa$ is $\frac12 d_\alpha^\kappa$. 
		The ``$2+\kappa$'' vertices of the original $D_1^\kappa$ in $D_{\alpha+1}^\kappa$ are denoted $t_{\alpha+1}$, $b_{\alpha+1}$, $x^i_{\alpha+1}$, $i<\kappa$. Observe that $t_{\alpha+1}=t_\alpha^{(j,+)}$ and $b_{\alpha+1}=b_\alpha^{(j,-)}$ for every $j<\kappa$. 
		Whenever we write $D_{\alpha+1}^\kappa=\bigcup_{j<\kappa} D_\alpha^{(j,+)} \cup D_\alpha^{(j,-)}$ (or similar expressions), we tacitly assume this identification. Then, $d_{\alpha+1}^\kappa$ being defined on $D_{\alpha+1}^\kappa$ as the shortest path metric, we still have  $d_{\alpha+1}^\kappa(t_{\alpha+1},b_{\alpha+1})=1$. 
		
		When $\alpha<\omega_1$ is a limit ordinal, we choose an enumeration $(\beta_n)_{n=0}^\infty$  of all ordinals less than $\alpha$. 
		For convenience and safety let us choose the same enumeration as the one in the definition of $T_\alpha$ or $S_\alpha$. 
		Then $D_\alpha^\kappa$ is the disjoint union of the $D_{\beta_n}^\kappa$, $n\in \omega$, except that we identify all the points $t_{\beta_n}$, resp. $b_{\beta_n}$, and denote them $t_\alpha$, resp. $b_\alpha$. The restriction of $d_{\alpha}^\kappa$ to $D_{\beta_n}^\kappa$ is $d_{\beta_n}^\kappa$. Then $d_{\alpha}^\kappa$ is extended to be the shortest path metric on $D_\alpha^\kappa$, and again $d_{\alpha}^\kappa(t_{\alpha},b_{\alpha})=1$. 
	\end{definition}
	
	\begin{remark}
		The metric spaces $(D_\alpha^\kappa,d_\alpha^\kappa)$ are countable and complete (see Proposition 2.7 in \cite{Basset2025}).
	\end{remark}
	
	\medskip
	In the spirit of the article \cite{Ostrovskii} by M.I. Ostrovskii, we shall define the notion of active pairs in $D_\alpha^\kappa$. 
	
	\begin{definition}\label{d:AP}
		We define the set of \emph{active pairs} as another graph structure on the vertices of $D_\alpha^\kappa$ as follows.
		\begin{itemize}
			\item $AP_0^\kappa$ has only one element: the pair $\{t_0,b_0\}$.
			\item $AP_1^\kappa=\set{E\subset D_1^\kappa:\cardinality{E}=2}$ (the complete graph on $D_1^\kappa$).
			\item $AP_{\alpha+1}^\kappa$ is the union of $AP_1^\kappa$, the set of active pairs in the original $D_1^\kappa$ in $D_{\alpha+1}^\kappa$, and the set of active pairs $AP_\alpha^{\kappa(j,\pm)}=AP_\alpha^\kappa$ in each of the ``$2\times \kappa$'' copies $D_\alpha^{\kappa(j,\pm)}$ of $D_\alpha^\kappa$. 
			\item Let $\alpha<\omega_1$ be a limit ordinal and $(\beta_n)_n$  be the sequence used in the construction of $D_\alpha^\kappa$. Then $AP_\alpha^\kappa$ is just $\bigcup_{n\in\omega} AP_{\beta_n}^\kappa$.
		\end{itemize}
		When there is no ambiguity, we use the shorter notation $AP_\alpha:=AP_\alpha^\kappa$.
	\end{definition}
	
	The following notion of sub-Lipschitz embedding has been called `partially bilipschitz embedding' in \cite{Ostrovskii}.
	
	\begin{definition} Let $X$ be a Banach space and $\alpha \in [0,\omega_1)$. We say that $D_\alpha^\kappa$ \emph{sub-Lipschitz embeds} into $X$ if there exist $A,B>0$ and a map $f: D_\alpha^\kappa \to X$ such that 
		\[Ad_\alpha(u,v) \le \|f(u)-f(v)\| \le Bd_\alpha(u,v),\ \ \text{for all}\ \{u,v\}\in AP_\alpha^\kappa.
		\]
		For $F\subset (0,\omega_1)$, we say that the family $(D_\alpha^\kappa)_{\alpha\in F}$ \emph{equi-sub-Lipschitz embeds} into $X$ if there exist $A,B>0$ and maps $f_\alpha: D_\alpha^\kappa \to X$, for all $\alpha\in F$, such that 
		\[Ad_\alpha(u,v) \le \|f(u)-f(v)\| \le Bd_\alpha(u,v),\ \ \text{for all}\ \{u,v\}\in AP_\alpha^\kappa.
		\]
	\end{definition}

	
	\section{Embedding dyadic diamond graphs and \texorpdfstring{\dat s}{delta-alpha-trees}}\label{s:dyadictrees}

	In this section we will be interested in diamond graphs and their active pairs with $\kappa = 2$. Given two elements $x \neq y$ of a metric space $(M,d)$ and a map $f$ defined on $M$  with values in a Banach space, it will be convenient to denote $f_{xy} \coloneqq \frac{f(x) - f(y)}{d(x,y)}$, in analogy with the notation $m_{xy}$ for molecules in Lipschitz-free spaces. 
	We start by studying the consequences of the sub-Lipschitz embeddability of a dyadic diamond graph into a Banach space for its dyadic-tree index.
	
	\begin{proposition}\label{p:DyadicDiamondToTree}
		Let $\alpha \in [0, \omega_1)$. Let $X$ be a Banach space and $f \colon \Da^2 \to X$ be a map satisfying, for some $A \in (0, 1]$ and all $u,v \in D_\alpha^2$ such that $\set{u,v}\in AP_\alpha^2$:
		\[
		A\,\da(u,v) \leq \norm{f(u)-f(v)} \leq \da(u,v).
		\]
		Then there exists an $A$-$\alpha$-tree $(x_s)_{s \in \Ta} \subset B_X$  whose root is $\ftba$. In particular, $D(X,A)\ge DT(X,A)>\alpha$.
	\end{proposition}
	
	\begin{proof}
		We proceed by induction on $\alpha$. The statement is clear for $\alpha = 0$ since $f_{t_0,b_0} \in B_X$, so assume it is true for all $\gamma <\alpha$.
		
		
		\underline{If $\alpha = \beta+1$:} The inequality $2\norm{\fmol{\xa{0},\xa{1}}} \geq 2A \da(\xa{0}, \xa{1}) = 2A$ implies that at least one of the following inequalities \[\norm{\fmol{\ta,\xa{0}} - \fmol{\xa{0},\ba}} \geq 2A\ \text{or}\ \norm{\fmol{\ta,\xa{1}} - \fmol{\xa{1},\ba}} \geq 2A
		\]
		holds, for example the last one. Applying the induction hypothesis to the restrictions of $2f$ on $\DapUn$ and $\DamUn$, we get that $\fmol{\ta,\xa{1}}$ and $\fmol{\xa{1},\ba}$ are respectively the roots $y_\varnothing$, $z_\varnothing$ of two $A$-$\beta$-trees $(y_s)_{s \in \Tb}$, $(z_s)_{s \in \Tb}$ included in the unit ball of $X$. We set $x_\varnothing = \ftba \in B_X$ and, for $s \in \Ta \setminus \set{\varnothing}$: if $s$ is of the form $0 \smallfrown t$ for some $t \in \Tb$, we set $x_s = y_t$. Otherwise, since $\alpha$ is a successor ordinal, $s$ is necessarily of the form $1 \smallfrown t$ for some $t \in \Tb$ and we set $x_s = z_t$. In particular, $x_0=\fmol{\ta,\xa{1}}$ and $x_1=\fmol{\xa{1},\ba}$. We thus obtain an $\alpha$-tree $(x_s)_{s \in \Ta}$. Moreover, $x_\varnothing = \frac12(x_0 + x_1)$ with $\norm{x_0 - x_1} \geq 2A$, and for $s \in \Ta \setminus \set{\varnothing}$, $x_s$ satisfies  the defining conditions of a $A$-$\alpha$-tree (see Definition \ref{def: dat}) by the hypothesis that $(y_s)_{s \in \Tb}$ and $(z_s)_{s \in \Tb}$ are $A$-$\beta$-trees. So $(x_s)_{s \in \Ta}$ is a $A$-$\alpha$-tree, and the statement is true for $\alpha = \beta + 1$.
		
		\underline{If $\alpha$ is a limit ordinal:} let $(\beta_n)_n$ be the sequence used in the definition of $D_\alpha^2$.  Applying the induction hypothesis to the restrictions of $f$ on $\Dbn^2$, we get that $\fmol{\tbn,\bbn}$ is the root $y_\varnothing^n$ of some $A$-$\beta_n$-tree $(y_s^n)_{s \in \Tbn}$ included in $B_X$. Let us set $x_\varnothing = \ftba$. Any $s \in \Ta \setminus \set{\varnothing}$ can be written $s = n \smallfrown t_n$ for some $n \in \omega$ and $t_n \in \Tbn$. For such an $s$, let us set $x_s = y_{t_n}^n$. We thus obtain an $\alpha$-tree $(x_s)_{s \in \Ta}$. The only new vertex is $x_\varnothing$ which satisfies the second defining condition of a $A$-$\alpha$-tree (see Definition \ref{def: dat}): $x_n = y_\varnothing^n = \fmol{\tbn,\bbn} = \ftba = x_\varnothing,\ \text{for all}\ n\in \omega.$
		The remaining vertices satisfy the defining conditions of a $A$-$\alpha$-tree by the hypothesis that $(y_s^n)_{s \in \Tbn}$ are $A$-$\beta_n$-trees. So $(x_s)_{s \in \Ta}$ is an $A$-$\alpha$-tree, and the statement is true for $\alpha$.
		
		The last part of the statement follows from Proposition \ref{prop: link between dentability index and d-a-t}
	\end{proof}
	
	The idea of active pairs will be crucial for obtaining a converse to the previous proposition and therefore a nonlinear characterization of the $\alpha$-tree property. 
	
	\begin{proposition}\label{p:TreeToDyadicDiamond}
		Let $\alpha \in [0,\omega_1)$ and assume that $X$ has the $\alpha$-tree property. Then there exist $f \colon \Da^2 \to X$ and $\delta \in (0, 1]$ such that for all $u,v \in D_\alpha^2$ with $\set{u,v}\in AP_\alpha^2$, we have
		\[
		\delta\,\da(u,v) \leq \norm{f(u)-f(v)} \leq \da(u,v).
		\] 
	\end{proposition}
	\begin{proof} Assume that $X$ has the $\alpha$-tree property. So there exists $\delta \in (0,\frac12]$ such that $B_X$ contains a $2\delta$-$\alpha$-tree. By rescaling and translating this tree, we deduce the existence of $(w_s)_{s\in T_\alpha}$ a $\delta$-$\alpha$-tree  in $B_X$ such that $\norm{w_s}\ge \frac12$, for all $s \in T_\alpha$. We now claim that there exists $f_\alpha\colon D_\alpha^2 \to X$ which satisfies:
		\begin{itemize}
			\item[a)] $\delta\, d_\alpha(u,v)\leq \norm{f(u)-f(v)}\leq d_\alpha(u,v)$ for all $\set{u,v}\in AP_\alpha^2$,
			\item[b)] $f_\alpha(t_\alpha)=w_\varnothing$,
			\item[c)] $f_\alpha(b_\alpha)=0$.
		\end{itemize}   
		
		We will prove the claim by induction.  The case $\alpha=0$ is obvious. We detail the case $\alpha=1$, as it will be used in the induction. So, it remains to define $f(x_1^i)=\frac12w_{(i)}$ for $i=0,1$. 
		We can check that 
		\begin{itemize}
			\item $\norm{f_{t_1b_1}}=\norm{f(t_1)-f(b_1)}=\norm{w_\varnothing} \in [\frac12,1]\subset [\delta,1]$.
			\item $\|f_{t_1x_1^i}\|=2\|f(t_1)-f(x_1^i)\|=\|w_{(j)}\| \in [\frac12,1]$ for $i \in \set{0,1}$ where $i+j=1$.
			\item $\|f_{x_1^ib_1}\|=2\|f(x_1^i)-f(b_1)\|=\|w_{(i)}\| \in [\frac12,1]$ for $i \in \set{0,1}$.
			\item $\|f_{x_1^0x_1^1}\|=\|f(x_1^1)-f(x_1^2)\|=\frac12\|w_{(0)}-w_{(1)}\| \in [\delta,1]$
		\end{itemize}
		This gives $\delta\, d_1(u,v)\leq \|f(u)-f(v)\|\leq d_1(u,v)$ for all $u\neq v \in D_1^2$.
		
		For $\alpha$ limit and $(\beta_n)_{n\in \omega}$ our fixed enumeration of all ordinals less than $\alpha$, the induction hypothesis yields a collection of maps $f_{\beta_n}\colon D_{\beta_n}^2\to X$ such that $f_{\beta_n}(b_{\beta_n})=0$, $f_{\beta_n}(t_{\beta_n})=w_{(n)}=w_\varnothing$ and such that each $f_{\beta_n}$ has distortion on active pairs in $AP_{\beta_n}^2$ as above, i.e. satisfies a). Then, we just define $f_\alpha\colon D_\alpha^2 \to X$ as follows: $f_\alpha(b_\alpha)=0$, $f_\alpha(t_\alpha)=w_\varnothing$, and $f_\alpha=f_{\beta_n}$ on $D_{\beta_n}^2 \setminus \{t_\alpha,b_\alpha\}$. 
		
		For $\alpha=\beta+1$, the induction hypothesis yields two maps $f_\beta^i \colon D_\beta^2 \to X$, $i \in \set{0,1}$, satisfying a) on active pairs in $AP_{\beta}^2$ and such that $f_\beta^i(b_\beta)=0$ and $f_\beta^i(t_\beta)=w_{(i)}$ for $i \in \set{0,1}$. 
		We define $f\restricted_{D_\alpha^{(i,-)}}=\frac12f_\beta^i$ for $i \in \set{0,1}$.
		Further we define $f\restricted_{D_\alpha^{(i,+)}}=\frac12(w_{(i)}+f_\beta^j)$ for $i \in \set{0,1}$ and $j$ such that $i+ j =1$. 
		Notice that, despite the double definition, $f$ is well defined on $x_\alpha^0$, $x_\alpha^1$, $b_\alpha$ and $t_\alpha$ and we have $f(b_\alpha)=0$, $f(t_\alpha)=\frac12(w_{(0)}+w_{(1)})=w_\varnothing$ and $f(x_\alpha^i)=w_{(i)}$, for $i\in \{0,1\}$. We can now compute the upper and lower estimates on active pairs.
		By the induction hypothesis and homogeneity, they will be the same as above on active pairs in $AP_{\beta}^{2(j,\pm)}$, $j\in \set{0,1}$. The computation for the case $\alpha=1$  will give us the estimates for active pairs from the original $D_1^2$ in the construction of $D_\alpha^2$. 
	\end{proof}
	
	Combining Propositions \ref{p:DyadicDiamondToTree} and \ref{p:TreeToDyadicDiamond} we can give the following metric characterization of the $\alpha$-tree property. 
	
	\begin{theorem}\label{t:AlphaTreeCharacterization}  Let $X$ be a Banach space and $\alpha \in (0,\omega_1)$. Then $X$ has the $\alpha$-tree property if and only if $D_\alpha^2$ sub-Lipschitz embeds into $X$
	\end{theorem}
	
	As a immediate consequence we get the following.
	
	\begin{corollary}\label{c:AlphaTreeRigidity}
		Let $X$ and $Y$ be Banach spaces and $\alpha \in (0,\omega_1)$. Assume that $Y$ bi-Lipschitz embeds into $X$ and $X$ fails the $\alpha$-tree property. Then $Y$ fails the $\alpha$-tree property. In other words if $Y$ bi-Lipschitz embeds into $X$, then $DT(Y)\le DT(X)$.  
	\end{corollary}
	
	Corollary \ref{c:AlphaTreeRigidity} is a refinement of the following Lipschitz rigidity result, due to M.~I.~Ostrovskii \cite{Ostrovskii}.
	\begin{theorem}\label{t:ITPRigidity}
		Let $X$ and $Y$ be Banach spaces. Assume that $X$ fails the Infinite Tree Property and that $Y$ bi-Lipschitz embeds into $X$. Then $Y$ fails the Infinite Tree Property.
	\end{theorem}
	
	We include a proof using our tools.
	\begin{proof} Obviously, a Banach space with the Infinite Tree Property contains a separable subspace (spanned by the infinite tree) with the Infinite Tree Property. So we may assume that $Y$ is separable. Replacing $X$ by the closed linear span of the range of the Lipschitz embedding, we may also assume that $X$ is separable. Since $X$ fails the Infinite Tree Property, we have, by Theorem \ref{t:ITP-DT(X)}, that $DT(X)<\omega_1$. So, by Corollary \ref{c:AlphaTreeRigidity}, $DT(Y)<\omega_1$, and applying the other implication in Theorem \ref{t:ITP-DT(X)}, we get that $Y$ fails the Infinite Tree Property.
	\end{proof}
	
	\begin{remark} Let us mention that for separable Banach spaces $X$ and $Y$ so that $Y$ bi-Lipschitz embeds into $X$, the fact that $D(Y)\le D(X)$ follows from classical differentiablity results. Indeed, either $X$ has the RNP so the bi-Lipschitz embedding admits a point of Gateaux differentiability and the Gateaux derivative induces a linear embedding of $Y$ into $X$, or $X$ fails the RNP and the inequality is trivial. The fact that the RNP is stable under Lipschitz embeddings is also a classical consequence of the separable determination of RNP and the above differentiablity argument. It is important to note that, due to the existence of the Bourgain-Rosenthal space $X_{BR}$ failing the RNP and the Infinite Tree Property, Corollary \ref{c:AlphaTreeRigidity} and Theorem \ref{t:ITPRigidity} cannot be obtained by a differentiability argument. 
	\end{remark}
	
	\begin{question}\label{q:embeddingbushes} We do not know a metric invariant, in the spirit of Theorem \ref{t:AlphaTreeCharacterization}, that characterizes the condition $D(X) \le \alpha$. In view of Section \ref{s:bushes}, a direction that one could pursue may be to define ad hoc weighted and finitely branching diamond graphs  of ordinal heights and study their sub-Lipschitz embeddings into Banach spaces. 
	\end{question}
	
	\begin{question} Let us now complete Remark \ref{r:FTP1}. Johnson and Schechtman proved in  \cite{JohnsonSchechtman} that a Banach space $X$ is not super-reflexive if and only if the family $(D_n^2)_{n\in \omega}$ equi-Lipschitz embeds into $X$. It follows from our results that it is also equivalent to the fact that the family $(D_n^2)_{n\in \omega}$ equi-sub-Lipschitz embeds into $X$ and to the fact that $D_\omega^2$ sub-Lipschitz embeds into $X$. We do not know whether $D_\omega^2$ bi-Lipschitz embeds into every non super-reflexive Banach space. 
	\end{question}

	
	\section{Embedding countably branching diamond graphs, \texorpdfstring{\dast s}{sprawling trees} and convex fragmentability index}\label{s:embeddingCountableDiamond}
	
	This section is entirely devoted to the study of the bi-Lipschitz and sub-Lipschitz embeddability of the countably branching diamond graphs $D_\alpha^\omega$ into Banach spaces. 
	
	\subsection{Szlenk index and countably branching diamond graphs} 
	
	The next proposition is later generalized to possibly non-dual spaces $X$ in Proposition~\ref{p:CountablyBranchingSprawling} where we get in fact the stronger conclusion $C(X)>\alpha$. However, we wish to mention this simpler argument based on weak$^*$ compactness.

	\begin{proposition}\label{p:CountablyBranchingConvexSzlenk}
		Let $Y$ be a Banach space, $A\in (0,1]$, $\alpha \in [0,\omega_1)$, and $f:D^\omega_\alpha \to X=Y^*$ be a map such that 
		for all $\set{u,v} \in AP^\omega_\alpha$,
		$$A \da(u,v)\le \|f(u)-f(v)\|\le \da(u,v).$$
		Then $\ftba \in c_{2A}^\alpha(B_{Y^*})$, where $c_\varepsilon$ denotes the convex-Szlenk derivation. In particular $Sz(Y)=Cz(Y)>\alpha$. 
	\end{proposition}

	\begin{proof} Since $D^\omega_\alpha$ is countable, the closed linear span $E$ of $f(D^\omega_\alpha)$ is separable and normed by a separable subspace $Z$ of $Y$. It follows that $E$ isometrically embeds into $Z^*$. Therefore, we may and do assume that $Y$ is separable. 
		
		We prove the statement by induction. It is clearly true for $\alpha=0$, as $\ftba \in B_X$.  So assume it is true for all $\beta <\alpha$. The property passes clearly to limit ordinals, so assume that $\alpha=\beta+1$. By induction hypothesis, we get that for all $i\in \omega$, $\fmol{\ta,\xa{i}}, \fmol{\xa{i},\ba} \in c_{2A}^\beta(B_{Y^*})$. By weak$^*$ compactness and passing to subsequences we may assume that $(\fmol{\ta,\xa{i}})_i$ and $(\fmol{\xa{i},\ba})_i$ are  weak$^*$ converging to $u$ and $v$ in $c_{2A}^\beta(B_{Y^*})$ respectively. So $f_{t_\alpha b_\alpha}=\frac12(\fmol{\ta,\xa{i}}+\fmol{\xa{i},\ba})=\frac12(u+v)$. 
		For all $i\neq j$, 
		$\set{\xa{i},\xa{j}}\in AP^\omega_\alpha$ and so
		$\|\fmol{\xa{i},\ba}-\fmol{\xa{j},\ba}\|=\|\fmol{\ta,\xa{i}} - \fmol{\ta,\xa{j}}\|\ge 2A$. 
		So $u,v \in s_{2A}(c_{2A}^\beta(B_{Y^*}))$ and $f_{t_\alpha b_\alpha}=\frac12(u+v)\in c_{2A}(c_{2A}^\beta(B_{Y^*}))=c_{2A}^\alpha(B_{Y^*})$. Thus $Cz(Y)>\alpha$. We recall that Theorem \ref{t:LPR} ensures that $Sz(Y)=Cz(Y)$.
	\end{proof}
	
	As a byproduct we can compare the embeddability of $D_n^\omega$ and $T_n^\omega$ or $T_\infty^\omega$, where $T_n^\omega$ is the countably branching tree of height $n\in \omega$ equipped with the hyperbolic distance (i.e its natural graph metric), and $T_\infty^\omega$ is the infinite countably branching hyperbolic tree. We refer to \cite{BaudierKaltonLancien} for precise definitions.
	\begin{corollary}\label{c:HyperTrees} Let $X$ be a Banach space. Assume that the family $(D_n^\omega)_{n\in \omega}$ equi-sub-Lipschitz embeds into $X^{**}$. Then $T_\infty^\omega$ bi-Lipschitz embeds into $X$.
	\end{corollary}
	
	\begin{proof} By Proposition \ref{p:CountablyBranchingConvexSzlenk}, there exists $\eps >0$ so that  $c_{\eps}^n(B_{X^{**}})\neq \emptyset$, for all $n\in \omega$. Then, using weak$^*$ compactness, we get that $c_{\eps}^\omega(B_{X^{**}})\neq \emptyset$. Thus $Sz(X^*)=Cz(X^*)>\omega$. It now follows from Theorem 2.6 in \cite{BaudierKaltonLancien} that $T_\infty^\omega$ bi-Lipschitz embeds into $X$. 
	\end{proof}
	
	\begin{remark} This corollary contrasts with Theorem 2.7 in \cite{LNOO2018} which asserts that the family $(T_n^\omega)_{n \in \N}$ does not equi-bi-Lipschitz embed into the family $(D_n^\omega)_{n\in \N}$. 
		
		On the other hand it is easy to see that $T_\infty^\omega$ 
		isometrically embeds into $\ell_1$, while the family $(D_n^\omega)_{n\in \omega}$ does not equi-sub-Lipschitz embed into $\ell_1$. The latter follows, for instance, from Proposition \ref{p:CountablyBranchingConvexSzlenk} and the fact that $\ell_1$ is isometric to $c_0^*$ while $Sz(c_0)=\omega$. 
		
		Note also that the assumption that $(D_n^\omega)_{n\in \omega}$ equi-sub-Lipschitz embeds into $X^{**}$ is strictly weaker than $(D_n^\omega)_{n\in \omega}$ equi-sub-Lipschitz embedding into $X$. Indeed, $(D_n^\omega)_{n\in \omega}$ does not equi-sub-Lipschitz embed into $\ell_1$, but equi-bi-Lipschitz embed into $L_1$ (see \cite{7} or Section \ref{Diamonds_L1}), which is isomorphic to a subspace of $\ell_1^{**}$ (remember that $\ell_\infty$ is isomorphic to $L_\infty$).
	\end{remark}
	

	\subsection{Sprawling trees and embeddings of countably branching diamond graphs}
	
	Recall that the active pairs in $D_\alpha^\omega$ are defined in Definition~\ref{d:AP}.
	The following is an analogue of Proposition~\ref{p:DyadicDiamondToTree}.
	
	\begin{proposition}\label{p:CountablyBranchingSprawling}
		Let $X$ be a Banach space, $\alpha\in [0,\omega_1)$, $A\in (0,1]$ and $f:D^\omega_\alpha \to X$ be a map such that 
		for all $\set{u,v} \in AP_\alpha^\omega$:
		$$A\da(u,v)\le \|f(u)-f(v)\|\le \da(u,v).$$
		Then $B_X$ contains a bounded $2A$-$\alpha$-sprawling tree whose root is $f_{t_\alpha b_\alpha}$. In particular $C(X)\ge ST(X)>\alpha$. 
	\end{proposition}
	
	\begin{proof}
		We proceed by induction on $\alpha$. For $\alpha=0$ there is nothing to prove, as $f_{t_\alpha b_\alpha}\in B_X$. Now let us assume that the claim is true for every $\gamma<\alpha$. 
		
		The case when $\alpha$ is a limit ordinal is clear, so let us assume that  $\alpha=\beta+1$. We  have that $f_{t_\alpha b_\alpha}=\frac12(f_{t_\alpha x^i_\alpha}+f_{x^i_\alpha b_\alpha})$ for all $i\in \omega$. Recall that $D_\alpha^{(i,\pm)}$ denotes the corresponding scaled-down copies of $D_\beta^\omega$ in $D_\alpha^\omega$.
		By the induction hypothesis applied to the restriction of $2f$ to $D_\alpha^{(i,+)}$ and to $D_\alpha^{(i,-)}$ we obtain ``$2\times \omega$''-many $2A$-$\beta$-sprawling trees $(y_s^{(i,\pm)})_{s\in S_\beta}$ in $B_X$ with roots $f_{t_\alpha x^i_\alpha}$ and $f_{x^i_\alpha b_\alpha}$. 
		Moreover, for every $i\neq j \in \omega$, we have $\|f_{t_\alpha x^i_\alpha}-f_{t_\alpha x^j_\alpha}\|\geq 2A$. 
		Therefore, if we define $x_\varnothing=f_{t_\alpha b_\alpha}$, $x_{(0,i)^\smallfrown s}=y_s^{(i,-)}$ and $x_{(1,i)^\smallfrown s}=y_s^{(i,+)}$ for all $i\in \omega$ and  $s \in S_\beta$, we obtain the desired $2A$-$\alpha$-sprawling tree. 
		
		The last part of the statement follows from Proposition \ref{p:C(X)_ST(X)}.
	\end{proof}

	We now turn to the converse. The proof  will be almost identical to the proof of Proposition \ref{p:TreeToDyadicDiamond}. 
	
	\begin{proposition}\label{p:subLipEmbeddingOmegaDiamond}
		Let $X$ be a Banach space and $\alpha \in[0,\omega_1)$. Assume that $X$ has the $\alpha$-sprawling tree property. Then there exist $f \colon D^\omega_\alpha \to X$ and $A \in (0, 1]$ such that for all $u,v \in D^\omega_\alpha$ with $\set{u,v}\in AP_\alpha^\omega$, we have
		\[
		A\,\da(u,v) \leq \norm{f(u)-f(v)} \leq \da(u,v).
		\] 
	\end{proposition}
	
	\begin{proof} Assume that $X$ has the $\alpha$-sprawling tree property. So there exists $\delta \in (0,1]$ such that $B_X$ contains a $2\delta$-$\alpha$-sprawling tree. By rescaling and translating this tree, we deduce the existence of $(w_s)_{s\in S_\alpha}$ a $\delta$-$\alpha$-sprawling tree  in $B_X$ such that $\norm{w_s}\ge \frac12$, for all $s \in S_\alpha$. We now claim that there exists $f_\alpha\colon D_\alpha^\omega \to X$ which satisfies:
		\begin{itemize}
			\item[a)] $\frac\delta2\, d_\alpha(u,v)\leq \norm{f(u)-f(v)}\leq d_\alpha(u,v)$ for all $\set{u,v}\in AP_\alpha^\omega$,
			\item[b)] $f_\alpha(t_\alpha)=w_\varnothing$,
			\item[c)] $f_\alpha(b_\alpha)=0$.
		\end{itemize}  
		We will prove the claim by induction. The case $\alpha=0$ is obvious. For $\alpha=1$, we  define $f(x_1^i)=\frac12w_{(0,i)}$ for every $i \in \omega$. 
		We can check that 
		\begin{itemize}
			\item $\|f_{t_1b_1}\|=\|f(t_1)-f(b_1)\| \in [\frac12,1]$.
			\item $\|f_{t_1x_1^i}\|=2\|f(t_1)-f(x_1^i)\|=2\|w_\varnothing-\frac12w_{(0,i)}\|=\|w_{(1,i)}\| \in [\frac12,1]$, $i \in \omega$.
			\item $\|f_{x_1^ib_1}\|=2\|f(x_1^i)-f(b_1)\|=\|w_{(0,i)}\| \in [\frac12,1]$, $i \in \omega$.
			\item $\|f_{x_1^ix_1^j}\|=\|f(x_1^1)-f(x_1^j)\|=\|\frac12(w_{(0,i)}-w_{(0,j)})\| \in (\frac\delta2,M]$
		\end{itemize}
		This gives $\frac\delta2\, d_1(u,v)\leq \norm{f(u)-f(v)}\leq  d_1(u,v)$ for all $\set{u, v} \in AP_1^\omega$.
		
		For $\alpha$ limit and $(\beta_n)_{n\in \omega}$ our fixed enumeration of all ordinals less than $\alpha$, the induction hypothesis yields a collection of maps $f_{\beta_n}\colon D_{\beta_n}^\omega \to X$ such that $f_{\beta_n}(b_{\beta_n})=0$, $f_{\beta_n}(t_{\beta_n})=w_{(n)}=w_\varnothing$ and such that each $f_{\beta_n}$ has distortion on active pairs in $AP_{\beta_n}^\omega$ satisfying a). Then, we just define $f_\alpha\colon D_\alpha^\omega \to X$ as follows: $f_\alpha(b_\alpha)=0$, $f_\alpha(t_\alpha)=w_\varnothing$, and $f_\alpha=f_{\beta_n}$ on $D_{\beta_n}^\omega \setminus \{t_\alpha,b_\alpha\}$. 
		
		For $\alpha=\beta+1$, the induction hypothesis yields  maps $f_\beta^{(0,i)}, f_\beta^{(1,i)}:D_\beta^\omega \to X$ satisfying a) on active pairs in $AP_\beta^\omega$ and such that  
		\[f_\beta^{(0,i)}(b_\beta)=f_\beta^{(1,i)}(b_\beta)=0\ \text{and}\  f_\beta^{(0,i)}(t_\beta)=w_{(0,i)},\ f_\beta^{(1,i)}(t_\beta)=w_{(1,i)},\ \text{for}\  i\in \omega.
		\]
		For $i \in \omega$, we now define $f_\alpha\restricted_{D_\alpha^{(i,-)}}:=\frac12f_\beta^{(0,i)}$  and $f_\alpha\restricted_{D_\alpha^{(i,+)}}:=\frac12(w_{(0,i)}+f_\beta^{(1,i)})$. 
		Notice that, despite the multiple definitions, $f_\alpha$ is well defined on $x_\alpha^i$, $b_\alpha$ and $t_\alpha$. In particular we have $f_\alpha(t_\alpha)=\frac12(w_{(0,i)}+w_{(1,i)})=w_\varnothing$. We can now estimate the upper and lower estimates on active pairs. By the induction hypothesis and homogeneity, they will be the same as above on active pairs in $AP_{\beta}^{\omega(j,\pm)}$, $j\in \omega$. The computation for the case $\alpha=1$  will give us the estimates for active pairs from the original $D_1^\omega$ in the construction of $D_\alpha^\omega$. 
	\end{proof}

	Combining Propositions \ref{p:CountablyBranchingSprawling} and \ref{p:subLipEmbeddingOmegaDiamond} we get the following metric characterization of the $\alpha$-sprawling tree property. 
	
	\begin{theorem}\label{t:AlphaSprawlingTreeCharacterization}  Let $X$ be a Banach space and $\alpha \in (0,\omega_1)$. Then $X$ has the $\alpha$-sprawling tree property if and only if $D_\alpha^\omega$ sub-Lipschitz embeds into $X$.
	\end{theorem}
	
	As an immediate consequence we get the following.
	
	\begin{corollary}\label{c:AlphaSprawlingTreeRigidity}
		Let $X$ and $Y$ be Banach spaces and $\alpha \in (0,\omega_1)$. Assume that $Y$ bi-Lipschitz embeds into $X$ and $X$ fails the $\alpha$- sprawling tree property. Then $Y$ fails the $\alpha$-sprawling tree property. In other words if $Y$ bi-Lipschitz embeds into $X$, then $ST(Y)\le ST(X)$.  
	\end{corollary}
	
	\begin{theorem}\label{t:ISTPRigidity}
		Let $X$ and $Y$ be Banach spaces. Assume that $X$ fails the Infinite Sprawling Tree Property and that $Y$ bi-Lipschitz embeds into $X$. Then $Y$ fails the Infinite Sprawling Tree Property.
	\end{theorem}
	
	\begin{proof} The proof is the same as for Theorem \ref{t:ITPRigidity}. 
	\end{proof}
	
	Let us conclude this section with a statement on universal spaces for bi-Lipschitz embeddings. 
	
	\begin{corollary}\label{c:obstructionToSPCP}
		Let $X$ be a separable Banach space such that every countable complete metric space bi-Lipschitz embeds into $X$. Then $X$ has the Infinite Sprawling Tree Property and thus fails the Slice PCP. In particular if $c_0$ bi-Lipschitz embeds into $X$, then $X$ fails Slice PCP.
	\end{corollary}
	
	\begin{proof}
		Assume that $X$ fails the Infinite Sprawling Tree Property. Then, by Theorem \ref{t:ISTP_ST(X)}, $ST(X)=\alpha<\omega_1$ and by the previous proposition, $D_\alpha^\omega$ does not bi-Lipschitz embeds into $X$. This is a contradiction. The fact that $X$ fails Slice PCP then follows from Propositions \ref{p:C(X)_ST(X)} and \ref{p:SlicePCP}. The last part of the statement follows from Aharoni's theorem \cite{Aharoni}, which states that every separable metric space bi-Lipschitz embeds into $c_0$.
	\end{proof}
	
	\begin{remark}
		The fact that a separable Banach space which is universal for separable metric spaces and bi-Lipschitz embeddings fails Slice PCP cannot be obtained through differentiability arguments and is a new result. Lemma \ref{l:LLT} is the tool that allowed us to go around the use of weak$^*$ compactness. We do not know whether we can avoid the use of convexity. The following questions remain open. 
	\end{remark}
	
	\begin{question}\label{q:universalPCP}
		Assume that $D_\alpha^\omega$ (sub)-Lipschitz embeds into $X$. Does it imply that $\Phi(X)>\alpha$? If $c_0$ bi-Lipschitz embeds into a separable Banach space $X$, does it imply that $X$ fails PCP?
	\end{question}
	
	\begin{remark}
		It is important to mention that a nice result on universal spaces has been recently obtained by R. J. Smith in \cite{Smith}, where it is proved that if a complete separable metric space $M$ contains a bi-Lipschitz copy of every countable complete and discrete metric space, then $M$ must contain a bi-Lipschitz copy of $c_0$. In particular, if $X$ is separable Banach space, then using a classical differentiability argument, it follows that $X$ cannot have RNP, or using Corollary \ref{c:obstructionToSPCP}, $X$ must fail Slice PCP. The arguments from \cite{Smith} are based on the use of the descriptive set theory for the class of complete separable metrics. This implies to work on a well chosen Polish space and is only adapted to the separable setting. On the other hand, we could define diamonds for any ordinal height. We shall not detail this generalization. Let us just mention that it allows us to prove the following result. Let $X$ be a Banach space with density character $\aleph$ such that every complete metric space of cardinality $\aleph$ bi-Lipschitz embeds into $X$. Then $X$ fails Slice PCP and thus also fails RNP.
	\end{remark}
	
	\subsection{Some remarks on asymptotic midpoint uniform convexity}
	
	First, we recall the notion of asymptotic midpoint uniform convexity.
	
	\begin{definition} Let $(X,\|\ \|)$ be a Banach space. The norm $\|\ \|$ is said to be \emph{asymptotically midpoint uniformly convex} (AMUC) if for every $\varepsilon>0$, there exists $\delta>0$ such that for all $x\in X$ and all weakly null net $(u_\alpha)_{\alpha \in A}$ in X satisfying $\|u_\alpha\|\ge \varepsilon$, $\|x+u_\alpha\|\le 1$ and $\|x-u_\alpha\|\le 1$ for all $\alpha \in A$, we have $\|x\| \le 1-\delta$. For a given $\eps>0$, the supremum of all $\delta \ge 0$ satisfying the above property is denoted $\delta_{AMUC}(\eps)$ and called the AMUC modulus. 
	\end{definition}
	
	\begin{remark}
		It is clear that an AUC norm is AMUC. There are examples of AMUC norms that are not AUC \cite{DKRRZ}. It is an important open question to know whether any space with an AMUC  norm admits an equivalent AUC norm.
	\end{remark}
	
	Let us introduce two properties of norms, that are well adapted to the subject of this article and will turn out to be equivalent to AMUC. 
	\begin{definition}  Let $(X,\|\ \|)$ be a Banach space.
		\begin{enumerate}
			\item We say that the norm $\|\ \|$ is  \emph{asymptotically sprawling uniformly convex} (ASUC) if for every $\varepsilon>0$, there exists $\delta>0$ such that for all $x\in B_X$ and all $\eps$-spider in $B_X$ rooted at $x$, we have that $\|x\| \le 1-\delta$. The supremum of all $\delta\ge 0$ satisfying the above is denoted $\delta_{ASUC}(\eps)$. \item We say that the norm $\|\ \|$ is  \emph{asymptotically diamond uniformly convex} (ADUC) if for every $\varepsilon>0$, there exists $\delta>0$ such that for all $x\in B_X$ we have the following implication: if there exists $f:D^\omega_1 \to B_X$ such that $f(b)=0$, $f(t)=x$ and $\varepsilon d(x,y)\leq \norm{f(x)-f(y)}\leq d(x,y)$ for all $x,y \in D_1^\omega$, then $\norm{x}\leq 1-\delta$. The supremum of all $\delta\ge 0$ satisfying the above is denoted $\delta_{ADUC}(\eps)$.
		\end{enumerate}
	\end{definition}
	
	The following proposition establishes the simple links between these moduli. 
	
	\begin{proposition}  Let $(X,\|\ \|)$ be a Banach space and $\eps >0$. Then
		\begin{enumerate}
			\item $\delta_{AMUC}(\eps)\ge \delta_{ASUC}(\eps/2)$ and $\delta_{ASUC}(\eps)\ge \delta_{AMUC}(\eps/2)$.
			\item $\delta_{ADUC}(\eps)\ge \delta_{ASUC}(2\eps)$ and
			$\delta_{ASUC}(\eps)\ge \delta_{ADUC}(\eps/4)$.
		\end{enumerate}
		In particular, the following are equivalent:
		\begin{enumerate}
			\item $\norm{\cdot}$ is ASUC,
			\item $\norm{\cdot}$ is ADUC,
			\item $\norm{\cdot}$ is AMUC.
		\end{enumerate}
	\end{proposition}
	
	\begin{proof} Let $x\in X$ and a weakly null net $(u_\alpha)_{\alpha}$ in X satisfying $\|u_\alpha\|\ge \varepsilon$, $\|x+u_\alpha\|\le 1$ and $\|x-u_\alpha\|\le 1$ for all $\alpha \in A$. Fix $\eta<\eps/2$. By a gliding hump argument we can build $(\alpha_n)_{n\in \N}\subset A$ such that $\|u_{\alpha_n}-u_{\alpha_m}\|\ge \eta$ and  write $x=\frac12(x-u_{\alpha_n}+x+u_{\alpha_n})$. We have built an $\eta$-spider in $B_X$ rooted at $x$. Therefore $\|x\|\le 1-\delta_{ASUC}(\eta)$. By continuity of $\delta_{ASUC}$, we deduce that $\delta_{AMUC}(\eps)\ge \delta_{ASUC}(\eps/2)$.
		
		Let $x\in B_X$ be the root of an $\eps$-spider in $B_X$. So, there exist $(y_n)_{n \in \N}$, $(z_n)_{n\in \N}$ in $B_X$ satisfying $x=\frac12(y_n+z_n)$ for all $n \in \N$ and $\|y_n-y_m\|\ge \varepsilon$ for all $n\neq m$. Let $\mathcal V$ be a weak neighborhood basis of $0$. By Lemma \ref{l:LLT}, there exist $(x^1_V)_{V \in \mathcal V}$ and $(x^2_V)_{V \in \mathcal V}$ nets in $B_X$ such that for all $V\in \mathcal V$, $x^1_v,x^2_V \in x+V$ and for any $V \in \mathcal V$ there exist $n\neq m$ so that $x^1_V=\frac12(y_n+z_m)$ and $x^2_V=\frac12(y_m+z_n)$. Then we can write $x^1_V=x+u_V$ and $x^2_V=x-u_V$, with $u_V=\frac12(z_m-z_n) \in V$ and $\|u_V\|\ge \varepsilon/2$. Therefore $\|x\|\le 1-\delta_{AMUC}(\eps/2)$. We have proved that $\delta_{ASUC}(\eps)\ge  \delta_{AMUC}(\varepsilon/2)$.
		
		The fact that $\delta_{ADUC}(\eps)\ge \delta_{ASUC}(2\eps)$ follows from Proposition \ref{p:CountablyBranchingSprawling}.
		
		The fact that $\delta_{ASUC}(\eps)\ge \delta_{ADUC}(\eps/4)$ follows from the proof of Proposition \ref{p:subLipEmbeddingOmegaDiamond}, in the case $\alpha =1$. 
	\end{proof}
	
	We introduce now yet another linear  invariant for Banach spaces. Due to the seminal paper of P. Enflo \cite{Enflo}, we know that a Banach space is super-reflexive if and only if it admits an equivalent uniformly convex norm if and only if it fails to have the Finite Tree Property. In the same spirit we give the following definition. 
	
	\begin{definition} We say that a Banach space $X$ has the \emph{Finite Sprawling Tree property} (FSTP) if there exists $\delta>0$ such that for every $n\in \omega$, $B_X$ contains a $\delta$-$n$-sprawling tree. 
	\end{definition}
	
	The following proposition is elementary.
	\begin{proposition}
		Let $X$ be a Banach space. Then $X$ has the FSTP if and only if there exists $\delta>0$ such that $B_X$ contains a $\delta$-$\omega$-sprawling tree.
	\end{proposition}
	
	\begin{proof} Assume that $X$ has FSTP. Then there exists $\delta>0$ such that, for all $n\in \omega$, $B_X$ contains a $2\delta$-$n$-sprawling tree: $(x_s^n)_{s\in S_n}$. Let $y_s^n=\frac12(x_s^n-x_\varnothing^n)$. Then $(y_s^n)_{s\in S_n}$ is a $\delta$-$n$-sprawling tree in $B_X$ with root $y_\varnothing=y_\varnothing^n=0$. Then, for $s\in S_\omega\setminus \{\varnothing\}$, there exists $n\in \omega$ and $t \in S_n$ such that $s=(0,n)\smallfrown t$. We set $y_s=y_t^n$. Clearly $(y_s)_{s\in S_\omega}$ is a $\delta$-$\omega$-sprawling tree in $B_X$.
		
		The converse implication is obvious. 
	\end{proof}
	
	Combining this with Propositions \ref{p:CountablyBranchingSprawling} and \ref{p:subLipEmbeddingOmegaDiamond}, we can state:
	
	\begin{proposition} Let $X$ be a Banach space. Then the following are equivalent.
		\begin{enumerate}
			\item $X$ has FSTP.
			\item There exists $\delta>0$ such that $B_X$ contains a $\delta$-$\omega$-sprawling tree.
			\item The family $(D_n^\omega)_{n \in \omega}$ equi-sub-Lipschitz embeds into $X$.
			\item $D_\omega^\omega$ sub-Lipschitz embeds into $X$. 
		\end{enumerate}
	\end{proposition}
	
	
	We now turn to the links with the existence of an equivalent AMUC norm. The following corresponds essentially to Theorem 4.1 in \cite{7}, with the difference that it is stated in terms of sub-Lipschitz instead of bi-Lipschitz embeddability.
	
	\begin{proposition}\label{prop:embedding in AMUC}
		Let $X$ be a Banach space.
		\begin{enumerate}
			\item Assume $X$ admits an equivalent AMUC norm. Then the family $(D_n^\omega)_{n \in \omega}$ does not equi-sub-Lipschitz embed into $X$.
			\item Assume that there exists $q\in [1,\infty)$ and $\gamma>0$ such that $\delta_{AMUC}(\eps)\ge \gamma\eps^q$. Then, there exists $c>0$ such that for any $n\in \omega$ and any $f:D_n^\omega \to X$, the distorsion of $f$ on active pairs of $D_n^\omega$ is at least $cn^{1/q}$.
		\end{enumerate}
	\end{proposition}
	\begin{proof}
		(1) Assume, as we may that the norm of $X$ is AMUC and therefore ADUC. Let $A\in (0,1)$ and $f:D_n^\omega \to X$ be a map such that 
		\[A d_n^\omega(u,v) \le \|f(u)-f(v)\|\le d_n^\omega(u,v),\ \ \text{for all}\ \{u,v\} \in AP_n^\omega.
		\]
		Then an easy induction shows that $\|f_{t_nb_n}\|\le (1-\delta_{ASUC}(A))^n$. But $\|f_{t_nb_n}\|\ge A$, so $k\le \ln(1/A)(\ln(1/1-\delta_{ASUC}(A))^{-1}$. This finishes the proof of (1). 
		
		Then (2) follows easily from the previous estimate and elementary calculus.
	\end{proof}
	
	We conclude this section with an important open question. 
	
	\begin{question}\label{q:ASUC_FSTP}
		Assume $X$ fails FSTP. Does $X$ admit an equivalent ASUC norm?
	\end{question}
	
	Let us explain what would be the implications of a positive answer and what could be a strategy of proof. It follows from the discussion in this section that the converse is true. Then we would have the following equivalences: a Banach space is AMUC renormable if and only if it fails FSTP if and only if the family $(D_n^\omega)_{n \in \omega}$ does not equi-sub-Lipschitz embed into $X$. Now recall that a map $f:X\to Y$ between Banach spaces is a \emph{coarse Lipschitz embedding} if there are constants $A,B,C>0$ such that for all $x,x'\in X$ such that $\|x-x'\|\ge C$, we have
	$$A\|x-x'\|_X\le \|f(x)-f(x')\|_Y\le B\|x-x'\|_X.$$
	Since the metric spaces $D_n^\omega$ are uniformly discrete, their non equi-sub-Lipschitz embeddability is stable under coarse Lipschitz embeddings. So a positive answer to Question \ref{q:ASUC_FSTP} would imply that the property of being AMUC renormable is invariant under coarse Lipschiptz embeddings.
	
	Very naturally, one strategy could be to adapt the proof by Enflo \cite{Enflo} of the fact that a Banach space failing the Finite Tree Property admits an equivalent uniformly convex norm.


	\section{\texorpdfstring{Embedding of the $\Da^\omega$ graphs into $L_1[0,1]$}{Embedding of the graphs into}}\label{Diamonds_L1}
	
	In this section, we define a ``limit diamond'' $\Dinf$ as the inductive limit of the countably branching diamond graphs $(D_n^\omega)_{n=0}^\infty$ of finite height. We show that $\Dinf$ embeds bi-Lipschitz into $L_1[0,1]$ with distortion $2$, and then that for every countable ordinal $\alpha$, $\Da^\omega$ embeds isometrically into $\Dinf$. We deduce that the family $(\Da^\omega)_{\alpha<\omega_1}$ equi-bi-Lipschitz embeds into $L_1[0,1]$ with distortion $2$.  
	
	\subsection{\texorpdfstring{The limit diamond $\Dinf$}{The limit diamond}}
	
	Inspired by Ostrovskii's definition of a ``limit binary diamond" in~\cite{Ostrovskii} (denoted $D_\omega$ in his paper), we define $\Dinf$ as the union of the vertex sets of $\{D_n^\omega\}_{n = 0}^{\infty}$. For every $u, v \in \Dinf$, there exists $n \in \omega$ such that $u, v \in D_n^\omega$: we set $\dinf(u,v) = d_n(u, v)$. Since the canonical embeddings $D_n^\omega \to D_{n+1}^\omega$ are isometric (see (12) and (13) in \cite{7} for an explicit description of these embeddings), the distance $\dinf$ is well-defined and does not depend on the choice of $n$, such that $u, v \in D_n^\omega$. We identify all the points $t_n$, resp. $b_n$, and denote the resulting point $t$, resp. $b$. Following the notation previously used for the diamond graphs, given $i \in \omega$, we denote by $D_\infty^{\omega(i,\pm)}$ (or $D_\infty^{(i,\pm)}$ for short) the subgraph of $\Dinf$ consisting of the union over $n \in \omega$ of the vertex sets of $D_n^{\omega(i,\pm)}$.
	
	It is shown in \cite{7} that for all $n \in \omega$, $D_n^\omega$ embeds into $L_1[0,1]$ with distortion 2. The purpose of this short section is to note that these embeddings are compatible and yield a bi-Lipschitz embedding of $\Dinf$ into $L_1[0,1]$. So we shall refer the reader to \cite{7} for the notation and details. For $n \in \omega$ and $x\in D_n^\omega$,  we adopt the same coding $x = (A, r)$ as in \cite[Section 2.2]{7}, where $A \in \omega^{\leq n}$ describes the branch $x$ belongs to in $D_n^\omega$ and $r=d_n(x, b_n)$ is a dyadic rational.
	
	\medskip 
	Let us briefly explain the following extension of \cite[Theorem 3.3]{7}. 
	
	\begin{theorem}
		\label{thm: Dinf into L1}
		There exists $\Psi \colon \Dinf \to L_1[0,1]$ such that, for every $x, y \in \Dinf$:
		\[
		\frac12\dinf(x,y) \leq \norm{\Psi(x) - \Psi(y)}_1 \leq \dinf(x,y).
		\]
	\end{theorem}
	
	\begin{proof}
		It has been proved in \cite[Theorem 3.3]{7} that, for every $k \in \omega$, there exists $\Phi_k \colon \widetilde{D}_k^\omega \to L_1[0,1]$ such that, for every $x, y \in \widetilde{D}_k^\omega$:
		\[
		\frac12\Tilde{d_k}(x,y) \leq \norm{\Phi_k(x) - \Phi_k(y)}_1 \leq \Tilde{d_k}(x,y)
		\]
		where $(\widetilde{D}_k^\omega, \Tilde{d_k})$ is the $k^{\text{th}}$ countably branching diamond graph, but built with diameter $2^k$ rather than $1$. In our context, $d_k = 2^{-k}\Tilde{d_k}$. The map $\Phi_k$ in \cite{7} is defined by $\Phi_k(x) \coloneqq 2^k \chi_{S_k(x)}$ where $S_k(x)$ is a specific union of subsets of $[0,1]$ associated to $x$. We rescale it to define our map $\Psi_k \colon D_k^\omega \to L_1[0,1]$ as $\Psi_k(x) \coloneqq \chi_{S_k(x)}$, which compensates for the rescaling of the metric. Thus $\Psi_k$ satisfies the same bi-Lipschitz inequalities with respect to $d_k$.
		
		We claim that for every $k$, $\Psi_{k+1}\restricted_{D_k^\omega} = \Psi_k$ where $D_k^\omega$ is seen as a subset of $D_{k+1}^\omega$ through the canonical isometry. Indeed, let $x \in D_k^\omega$ and let us examine the definition of $S_k(x)$ for $x$ seen as an element of $D_{k+1}^\omega$ via the canonical inclusion. Following the notation of \cite{7}, if $x = (A, r)$ in $D_k^\omega$, its representation in $D_{k+1}^\omega$ is identical. The construction of the sets $S_k(x)$ relies on sets $T_k^i(x)$ which depend only on $A$ and $r$ (see \cite[Section 3.3]{7}), and therefore are stable under this inclusion: $T_{k+1}^i(x) = T_k^i(x)$ and thus $S_{k+1}(x) = S_k(x)$. The maps $\Psi_k$, $k \in \omega$, are therefore compatible, and the map $\Psi \colon \Dinf \to L_1[0,1]$ defined by $\Psi(x) \coloneqq \Psi_k(x)$ for $k \in \omega$ such that $x \in D_k^\omega$ is well-defined and satisfies the desired bi-Lipschitz inequalities.
	\end{proof}

	\subsection{\texorpdfstring{Isometric embedding of $\Da^\omega$  into $\Dinf$}{Isometric embedding of into}}
	First, let us list, without proof, a few lemmas, which follow immediately from the definitions of $D_\alpha^\omega$ and $\Dinf$. 
	
	\begin{lemma}
		\label{lemma: distances in D_alpha, alpha successor}
		Let $\alpha = \beta + 1 \in (0, \omega_1)$ be a successor ordinal and $x, y \in \Da^\omega \setminus \{\ta, \ba\}$: there exist $i, j \in \omega$ such that $x \in \Da^i \coloneqq \Da^{(i,+)} \cup \Da^{(i,-)}$ and $y \in \Da^j \coloneqq \Da^{(j,+)} \cup \Da^{(j,-)}$. Then:
		\begin{enumerate}
			\item If $i = j$ and $x, y$ both belong to the same $\Da^{(i,\pm)}$, then $\da(x,y) = \frac{1}{2} \db(x,y)$. Here, by $\db(x,y)$, we abusively mean the distance in $D_\beta^\omega$ of their images by the canonical bijection from  $\Da^{(i,\pm)}$ onto $D_\beta^\omega$.
			
			\item If $i = j$ and $x \in \Da^{(i,+)}$, $y \in \Da^{(i,-)}$, then 
			\[
			\da(x,y) = \da(x, \ba) - \da(y, \ba)=\da(y,t_\alpha)-\da(x,t_\alpha).
			\]
			
			\item If $i \neq j$, then
			\[
			\da(x, y) = \left\{
			\begin{array}{ll}
				\da(x, \ba) + \da(y, \ba) & \mbox{if $\da(x, \ba) + \da(y, \ba) \leq 1$} \\
				2-\da(x,\ba)-\da(y,\ta)& \mbox{if $\da(x, \ba) + \da(y, \ba) \geq 1$}
			\end{array}.
			\right.
			\]
		\end{enumerate}
	\end{lemma}
	
	We observe similar identities in the limit ordinal case.
	
	\begin{lemma}
		\label{lemma: distances in D_alpha, alpha limit}
		Let $\alpha \in (0, \omega_1)$ be a limit ordinal and $(\beta_n)_{n \in \omega}$ be an enumeration of all ordinals less than $\alpha$. Let $x, y \in \Da^\omega = {\bigcup\limits_{n \in \omega}} \{n\} \times \Dbn^\omega$: there exists $n, m \in \omega$ and $x_n \in \Dbn^\omega$, $y_m \in D_{\beta_m}^\omega$ such that $x = (n, x_n)$ and $y = (m, y_m)$. Then:
		\begin{enumerate}
			\item If $n = m$, then $\da(x,y) = d_{\beta_n}(x_n, y_n)$.
			
			\item If $n \neq m$, then 
			\[
			\da(x, y) = \left\{
			\begin{array}{ll}
				d_{\beta_n}(x_n, \ba) + d_{\beta_m}(y_m, \ba) & \mbox{if $d_{\beta_n}(x_n, \ba) + d_{\beta_m}(y_m, \ba) \leq 1$} \\
				2 - d_{\beta_n}(x_n, \ba) - d_{\beta_m}(y_m, \ba) & \mbox{if $d_{\beta_n}(x_n, \ba) + d_{\beta_m}(y_m, \ba) \geq 1$}
			\end{array}.
			\right.
			\]
		\end{enumerate}
	\end{lemma}
	
	Letting $k$ tend to $\infty$ in the case $\alpha = k \in \omega$ in Lemma \ref{lemma: distances in D_alpha, alpha successor}, we deduce the distance formulas in $\Dinf$. 
	
	\begin{lemma}
		\label{lemma: distances in D_infinity}
		Let $x, y \in \Dinf$ and $i, j \in \omega$ such that $x \in D_\infty^i \coloneqq D_\infty^{(i,+)} \cup D_\infty^{(i,-)}$ and $y \in D_\infty^j$. Then: 
		\begin{enumerate}
			\item If $i = j$ and $x, y$ both belong to the same $D_\infty^{(i,\pm)}$, then there exists $k\in \N$ such that they belong to the same  $D_k^{(i,\pm)}$. Then $\dinf(x,y) = \frac12d_{k-1}(x,y)$, where again, elements in $D_k^{(i,\pm)}$ are identified with their canonical images in $D_{k-1}^{(i,\pm)}$
			
			\item If $i = j$ and $x \in D_\infty^{(i,+)}$, $y \in D_\infty^{(i,-)}$, then: 
			\[
			\dinf(x,y) =  d_\infty(x, b) - d_\infty(y, b)=d(y,t)-d(x,t). 
			\]
			
			\item If $i \neq j$, then
			\[
			\dinf(x,y) = \left\{
			\begin{array}{ll}
				\dinf(x, b) + \dinf(y, b) & \mbox{if $\dinf(x, b) + \dinf(y, b) \leq 1$} \\
				2 - \dinf(x, b) - \dinf(y, b) & \mbox{if $\dinf(x, b) + \dinf(y, b) \geq 1$}
			\end{array}.
			\right.
			\]
		\end{enumerate}
	\end{lemma}
	
	The key argument in the proof will be the property of self-similarity of $\Dinf$: replacing every $D_\infty^{(i,\pm)}$ by a $\frac12$-scaled down copy of $\Dinf$ yields a metric space isometric to $\Dinf$ itself. More precisely, we have:
	
	\begin{lemma}
		\label{lemma: def of g}\ 
		
		\begin{enumerate}
			\item 
			Given $i \in \omega$, the maps
			\[
			g^{(i,-)} \colon \left \{
			\begin{array}{lll}
				\Dinf & \to & \Dinfim \\
				x = (A, r) & \mapsto & (i \smallfrown A, \frac{r}{2})
			\end{array}
			\right.
			\]
			and
			\[
			g^{(i,+)} \colon \left \{
			\begin{array}{lll}
				\Dinf & \to & \Dinfip \\
				x = (A, r) & \mapsto & (i \smallfrown A, \frac{r+1}{2})
			\end{array}
			\right.
			\]
			are surjective isometries onto the respective subdiamonds equipped with the scaled down metric and preserving poles. More precisely: for every $x, y \in \Dinf$,
			\[
			\dinf(g^{(i,\pm)}(x),g^{(i,\pm)}(y)) = \frac{1}{2}\dinf(x,y).
			\]
			\item For $i\neq j \in \omega$, there exists an isometry from $D_\infty^{i}$ onto $D_\infty^{j}$, which preserves the poles. 
		\end{enumerate}
		
	\end{lemma}
	
	
	We now state and prove our embedding result.
	
	\begin{theorem}\label{t:UniversalityOfDinfty}
		For every $\alpha \in (0, \omega_1)$, there exists an isometry $\Psi_\alpha \colon \Da^\omega \to \Dinf$ such that $\Psi_\alpha(t_\alpha)=t$ and $\Psi_\alpha(b_\alpha)=b$. 
	\end{theorem}
	
	\begin{proof}
		We will prove it by induction on $\alpha$. This is clearly true for $\alpha=0$. Let $\alpha \in (0,\omega_1)$ and assume our statement holds for all $\mu < \alpha$. 
		
		\underline{Assume first that $\alpha = \beta + 1$.} For every $i \in \omega$, let $f_\alpha^{(i,\pm)} \colon \Da^{(i,\pm)} \to \Db^\omega$ be the canonical isometry between $(\Da^{(i,\pm)}, \da)$ and $(\Db^\omega, \frac{1}{2}\db)$. We define $\Psi_\alpha$ piecewise on each subdiamond as follows:
		\[
		\Psi_\alpha(x) \coloneqq g^{(i,\pm)} \circ \Psi_\beta \circ f_\alpha^{(i,\pm)}(x) \quad \text{ for } x \in \Da^{(i, \pm)},
		\]
		where $g^{(i,\pm)}$ has been defined in Lemma \ref{lemma: def of g} and $\Psi_\beta$ is the isometry provided by the induction hypothesis. It is a straightforward computation to check that:
		\begin{enumerate}
			\item[-] for every $i \in \omega$, $\Psi_\alpha(\xa{i})$ is well-defined, since we get the same value whether we consider $\xa{i}$ as an element of $\Da^{(i, +)}$ or $\Da^{(i, -)}$. Similarly, $\Psi_\alpha(\ta) = t$ and $\Psi_\alpha(\ba) = b$;
			
			\item[-] for every $z \in \Da^\omega$, $\dinf(\Psi_\alpha(z), b) = \da(z, \ba)$ and $\dinf(\Psi_\alpha(z), t) = \da(z, \ta)$; 
			
			\item[-] on each subdiamond $\Da^{(i,\pm)}$, $\Psi_\alpha$ is an isometry. 
		\end{enumerate}
		
		It remains to show that given $x, y \in \Da^\omega$, if $i, j \in \omega$ and $\varepsilon, \varepsilon' \in \set{+, -}$ are such that $x \in \Da^{(i,\varepsilon)}$ and $y \in \Da^{(j,\varepsilon')}$ with $i \neq j$ or $\varepsilon \neq \varepsilon'$, we have $\dinf(\Psi_\alpha(x), \Psi_\alpha(y)) = \dinf(x, y)$. This follows from the computation of the distances in $\Da^\omega$ and in $\Dinf$ (Lemmas \ref{lemma: distances in D_alpha, alpha successor} and \ref{lemma: distances in D_infinity}) and from the fact that $\Psi_\alpha$ preserves the distance to each pole.
		

		\underline{Assume now that $\alpha$ is a limit ordinal.} Let $(\beta_n)_{n \in \omega}$ be an enumeration of all ordinals less that $\alpha$. Write $\Da^\omega=\cup_{n\in \omega}\{n\}\times D_{\beta_n}^\omega$, identifying all $(n,t_{\beta_n})$ with $t_\alpha$ and all $(n,b_{\beta_n})$ with $b_\alpha$. By induction hypothesis, for every $n \in \omega$, there exists $\Psi_n \colon \{n\}\times D_{\beta_n}^\omega \to \Dinf$ such that $\Psi_n(n,t_{\beta_n})=t$ and $\Psi_n(n,b_{\beta_n})=b$. We now introduce the metric space $\Delta_\infty^\omega= \cup_{n\in \omega}\{n\}\times \Dinf$, identifying all $(n,t)$ with a point called $T$ and all $(n,b)$ with a point called $B$. Then the metric on  $\Delta_\infty^\omega$ is defined following exactly the same procedure as for $\Da^\omega$. For $n\in \omega$ and $x \in D_{\beta_n}^\omega$, we set $\Phi_\alpha(n,x)=\Psi_n(n,x)$. It is clear that $\Phi_\alpha$ is an isometry from $\Da^\omega$ into $\Delta_\infty^\omega$ so that $\Phi_\alpha(t_\alpha)=T$ and $\Phi_\alpha(b_\alpha)=B$. Consider now $\Delta_\infty^\omega$ as $\cup_{n,i\in \omega}\{n\}\times D_\infty^{i}$, where all the top (resp. bottom) points of the $\{n\}\times D_\infty^{i}$ are identified, and the metric defined as usual. Then, using a bijection from $\omega \times \omega$ onto $\omega$ and item (2) from Lemma \ref{lemma: def of g} we get that there is an isometry $\Psi$ from  $\Delta_\infty^\omega$ onto $\Dinf$ such that $\Psi(T)=t$ and $\Psi(B)=b$. The map $\Psi_\alpha=\Psi \circ \Phi_\alpha$ is the desired isometry. This completes the induction.
	\end{proof}

	Combining this with the embedding of $\Dinf$ into $L_1[0,1]$ (Theorem \ref{thm: Dinf into L1}), it follows:
	
	\begin{corollary}
		For every $\alpha \in (0, \omega_1)$, there exists a map $\Psi_\alpha \colon \Da^\omega \to L_1[0,1]$ such that, for every $x, y \in \Da^\omega$:
		\[
		\frac12\da(x,y) \leq \norm{\Psi_\alpha(x) - \Psi_\alpha(y)}_1 \leq \da(x,y).
		\]
	\end{corollary}
	
	We can now refine  the last sentences in Corollary \ref{c:obstructionToSPCP} and Question \ref{q:universalPCP}.
	
	\begin{corollary} Let $X$ be a Banach space. Assume that $L_1$ bi-Lipschitz embeds into $X$. Then $X$ fails the Slice PCP.
	\end{corollary}
	
	\begin{question}
		Let $X$ be a Banach space and assume that $L_1$ bi-Lipschitz embeds into $X$. Does it imply that $X$ fails the PCP?
	\end{question}
	
	In relation with these questions, we recall that if $c_0$ bi-Lipschitz embeds into a Banach space $X$, then $X$ must contain linearly and uniformly the $\ell_\infty^n$, $n\in \omega$ (or equivalently must have trivial cotype). Since $L_1$ has cotype 2, $c_0$ does not bi-Lipschitz embed into $L_1$. 

	\subsection{Concluding remarks}
	
	The definition of the limit dyadic diamond, due to Ostrovskii \cite{Ostrovskii}, that we will denote $D_\infty^2$ in this paper, is completely analogous to the definition of $D_\infty^\omega$: just consider the inductive limit of $D_n^2$ instead of $D_n^\omega$. Then, for $\kappa=2$, or $\kappa=\omega$, we can define the active pairs in $D_\infty^\kappa$ as $AP_\infty^\kappa=\bigcup_{n=0}^\infty AP_n^\kappa$. Then, using $AP_\infty^\kappa$, we can define the notion of a sub-Lipschitz embedding of $D_\infty^\kappa$. Combining our results with \cite{Ostrovskii}, we can complete our non linear characterizations of the Infinite Tree Property and the Infinite Sprawling Tree Property.
	
	\begin{theorem} Let $X$ be a separable Banach space. The following assertions are equivalent.
		\begin{enumerate}
			\item For all $\alpha<\omega_1$, $D_\alpha^2$ sub-Lipschitz embeds into $X$.
			\item For all $\alpha<\omega_1$, $X$ has the $\alpha$-tree property.
			\item $X$ has the Infinite Tree Property.
			\item $D_\infty^2$ sub-Lipschitz embeds into $X$.
		\end{enumerate}
	\end{theorem}
	
	\begin{proof}
		The equivalence $(1) \Leftrightarrow (2)$ is Theorem \ref{t:AlphaTreeCharacterization}, the equivalence $(2) \Leftrightarrow (3)$ is Theorem \ref{t:ITP-DT(X)}, and the equivalence $(3) \Leftrightarrow (4)$ is Theorem 1.12 in \cite{Ostrovskii}. 
	\end{proof}
	
	Similarly, we have. 
	\begin{theorem} Let $X$ be a separable Banach space. The following assertions are equivalent.
		\begin{enumerate}
			\item For all $\alpha<\omega_1$, $D_\alpha^\omega$ sub-Lipschitz embeds into $X$.
			\item For all $\alpha<\omega_1$, $X$ has the $\alpha$-sprawling-tree property.
			\item $X$ has the Infinite Sprawling Tree Property.
			\item $D_\infty^\omega$ sub-Lipschitz embeds into $X$.
		\end{enumerate}
	\end{theorem}   
	
	\begin{proof}
		The equivalence $(1) \Leftrightarrow (2)$ is Theorem \ref{t:AlphaSprawlingTreeCharacterization} and the equivalence $(2) \Leftrightarrow (3)$ is Theorem \ref{t:ISTP_ST(X)}. Finally, the  proof of the equivalence $(3) \Leftrightarrow (4)$ is an adaptation of the proof of Theorem 1.12 in \cite{Ostrovskii}, that we shall not detail here. 
	\end{proof}
	
	\section*{Acknowledgments}
	
	The three authors were partially supported by the French ANR project ``Comop'' ANR-24-CE40-0892-01. The LmB receives support from the EIPHI Graduate School (contract ANR-17-EURE-0002).


\end{document}